\documentclass{amsart}
\usepackage[numbers,sort&compress]{natbib}
\usepackage[euler-digits]{eulervm}
\usepackage{tikz}
\usetikzlibrary{snakes,arrows,shapes}
\usepackage[matrix,arrow,curve]{xy}
\newcommand{\f}[1]{\underline{#1}}
\newcommand{\s}[1]{\overline{#1}}
\renewcommand{\pmod}[1]{\left(\mathsf{mod}\;#1\right)}
\DeclareMathOperator{\p}{\mathsf{P}}
\DeclareMathOperator{\F}{\mathsf{F}}
\DeclareMathOperator{\E}{\mathsf{E}}
\DeclareMathOperator{\J}{\mathsf{j}}
\newcommand{\sym}[1]{\mathfrak{S}_{#1}}
\let\ker\relax\DeclareMathOperator{\ker}{\mathsf{ker}}

\DeclareMathOperator{\irr}{\mathsf{Irr}}
\DeclareMathOperator{\rad}{\mathsf{Rad}}

\renewcommand{\mid}{\;\middle\vert\;}
\definecolor{red}{rgb}{1,0,0}
\newcommand{\bb}[2]{\genfrac{\langle}{\vert}{0pt}{}{#1}{#2}}
\newcommand{\lb}[2]{\genfrac{\langle}{.}{0pt}{}{#1}{#2}}
\newcommand{\mb}[2]{\genfrac{}{}{0pt}{}{#1}{#2}}
\newcommand{\rb}[2]{\genfrac{.}{\vert}{0pt}{}{#1}{#2}}

\usepackage{aliascnt}
\usepackage[colorlinks=true,linkcolor=blue]{hyperref}
\def\equationautorefname~#1\null{(#1)\null}
\def\itemautorefname~#1\null{(#1)\null}
\def\sectionautorefname~#1\null{\S#1\null}
\newtheorem{theorem}{Theorem}  
\newaliascnt{proposition}{theorem}  
\newtheorem{proposition}[proposition]{Proposition}  
\aliascntresetthe{proposition}  
  
\newaliascnt{corollary}{theorem}  
\newtheorem{corollary}[corollary]{Corollary}  
\aliascntresetthe{corollary}  
  
\newaliascnt{lemma}{theorem}  
\newtheorem{lemma}[lemma]{Lemma}  
\aliascntresetthe{lemma}  
  
\newaliascnt{conjecture}{theorem}  
\newtheorem{conjecture}[conjecture]{Conjecture}  
\aliascntresetthe{conjecture}  

\theoremstyle{definition}
\newaliascnt{remark}{theorem}  
  
\aliascntresetthe{remark}

\title[Quiver Presentation Descent Symmetric]
{On the Quiver Presentation of the Descent Algebra of the Symmetric Group} 
\author{Marcus Bishop}
\address{Ruhr-Universit\"at Bochum\\Fakult\"at f\"ur Mathematik}
\email{marcus.bishop@rub.de}
\author{G\"otz Pfeiffer}
\address{National University of Ireland, Galway\\
School of Mathematics, Statistics, and Applied Mathematics}
\email{goetz.pfeiffer@nuigalway.ie}

\begin{document}
\begin{abstract}
  We describe a presentation of the descent algebra of the symmetric
  group $\sym{n}$ as a quiver with relations.
  This presentation arises from a new construction of the
  descent algebra as a homomorphic image of an
  algebra of forests of binary trees, which can be identified with
  a subspace of the free Lie algebra.
  In this setting we provide
  a short new proof of the known fact that the quiver of the descent
  algebra of $\sym{n}$ is given by restricted partition refinement.
  Moreover, we describe certain families of relations and
  conjecture that for fixed $n\in\mathbb{N}$ the finite set of relations
  from these families that are relevant for the descent algebra of
  $\sym{n}$ generates the ideal of relations of
  an explicit quiver presentation of that algebra.
\end{abstract}

\maketitle

\section{Introduction}\label{IntroductionSection}
Let $\left(W,S\right)$ be a finite Coxeter system and
let $k$ be a field of characteristic zero.
For all $J\subseteq S$
we denote the parabolic subgroup
$\left\langle J\right\rangle$ of $W$ by $W_J$
and the set of minimal length left coset representatives
of $W_J$ in $W$ by $X_J$.
In 1976 Solomon proved \cite{solomon} that
the elements $x_J=\sum_{x\in X_J} x\in kW$ with $J\subseteq S$ satisfy
\begin{equation}\label{SolomonTheorem}
x_Jx_K=\sum_{L\subseteq S}c_{JKL}x_L
\end{equation}
for certain integers $c_{JKL}$ with $J,K,L\subseteq S$.
This implies that the linear span
$\left\langle x_J\mid J\subseteq S\right\rangle$
is a {\em subalgebra} of $kW$. This algebra is called the 
{\em descent algebra} of $W$ and is denoted by $\Sigma\left(W\right)$.

Solomon shows \cite{solomon} that
the structure constants $c_{JKL}$ in \autoref{SolomonTheorem} are the same
constants appearing in the Mackey formula
for the product of the permutation characters
$\mathsf{Ind}^W_{W_J}1$ and $\mathsf{Ind}^W_{W_K}1$
in terms of the characters $\mathsf{Ind}^W_{W_L}1$ with $L\subseteq S$.
Therefore the map $\theta:\Sigma\left(W\right)\to
k\irr\left(W\right)$ given by
$x_J\mapsto \mathsf{Ind}^W_{W_J}1$ for all $J\subseteq S$
is a homomorphism of $k$-algebras,
where $k\irr\left(W\right)$ is the character ring of $W$ over $k$.
Solomon also shows that $\ker\theta$ is the
radical of $\Sigma\left(W\right)$.

We identify $k\irr\left(W\right)$ with the ring $k^m$
under pointwise addition and multiplication, where $m$ is the number
of conjugacy classes of $W$.
Then the map $\theta$ above presents the semisimple algebra
$\Sigma\left(W\right)/\mathop{\mathsf{Rad}}\Sigma\left(W\right)$
as a subalgebra of $k^m$.
Since $k^m$ is commutative, the simple $\Sigma\left(W\right)$-modules
are all one-dimensional over $k$.
This means that $\Sigma\left(W\right)$ is a basic algebra
which thereby admits a quiver presentation. See \cite{blue} for more
information about basic algebras and quivers.
The preceding discussion also shows that we can
assume that $k$ is the field $\mathbb{Q}$ of rational numbers,
since the permutation characters 
$\mathsf{Ind}^W_{W_J}$ take values in $\mathbb{Z}$.

The aim of this paper is to calculate and study the quiver
presentation of $\Sigma\left(W\right)$ when
$W$ is the symmetric group $\sym{n}$ of degree $n\ge 0$.
An elementary proof of formula \autoref{SolomonTheorem} in this case
was given  by Atkinson \cite{atkinson} in 1986.
The Coxeter generating set of $\sym{n}$ is
$S=\left\{1,2,\ldots,n-1\right\}$ where we identify each
$s\in S$ with the transposition exchanging
the points $s$ and $s+1$.
We remark that the set $X_J$ has a description
in terms of the graphs of the elements of $\sym{n}$.
Here we regard $w\in\sym{n}$ as a function from $\left\{1,2,\ldots,n\right\}$
to itself and the graph of $w$ as the set of points
$\left\{\left(i,i.w\right)\mid 1\le i\le n\right\}$.
Then $X_J$ is the set of all $w\in\sym{n}$ such that
$i.w>\left(i+1\right).w$ holds only
when $i\in J$. In other words, the only points where
the graph of $w$ is {\em descending} are in $J$.
The name {\em descent algebra} derives from this interpretation.

The algebra $\Sigma\left(\sym{n}\right)$ plays a major role in the
book by Blessenohl and Schocker \cite{ncct}
where the authors study the character theory of $\sym{n}$ through
an extension of the map $\theta$ above to $k\sym{n}$.
As in \cite{ncct}, this article takes the point
of view of studying $\Sigma\left(\sym{n}\right)$
for all $n\ge 0$ simultaneously
by uniting families of objects indexed by $n$ into single objects.
The industry of studying
$\Sigma\left(\sym{n}\right)$ through its quiver presentation 
begins in 1989 with Garsia and Reutenauer's description \cite{decomposition} of
the quiver of $\Sigma\left(\sym{n}\right)$.
We derive this quiver in \autoref{QuiverSection} using an algebra $k\mathcal{L}_n$
that we describe below.
Garsia and Reutenauer also calculate the Cartan invariants
and the projective indecomposable modules of $\Sigma\left(\sym{n}\right)$.
Aktinson \cite{revisited} derives these using elementary methods.

Bergeron and Bergeron \cite{quiverb1,quiverb2} partially describe
the quiver of $\Sigma\left(W\right)$ for $W$ of type~$B_n$  in 1992 
with their calculation of the idempotents of $\Sigma\left(W\right)$,
which correspond with the vertices of the quiver.
The full quiver in type~$B_n$ was calculated by Saliola \cite{quiverb3} 
in 2008 using hyperplane arrangements.

In a somewhat different direction, but amounting to essentially
the same information as a quiver presentation, the {\em module structure}
of $\Sigma\left(\sym{n}\right)$ was calculated
\cite{module1,module2} and later expanded by Schocker \cite{quivera},
where he showed that articles \cite{module2} and \cite{module1}
essentially calculate the quiver of $\Sigma\left(\sym{n}\right)$.
One component of the module structure of $\Sigma\left(W\right)$
is the length of its Loewy series,
which was calculated for $W$ of type $D_n$ and~$n$ odd
by Saliola \cite{loewy2} in 2010 after the calculation
by Bonnaf\'e and Pfeiffer in 2008 \cite{loewy1} for the
remaining finite irreducible Coxeter groups.

A first step towards the calculation
of the quiver for arbitrary Coxeter groups lies in
Bergeron, Bergeron, Howlett, and Taylor's
calculation \cite{idempotents} of a basis of idempotents
of $\Sigma\left(W\right)$
for any Coxeter group $W$, since these idempotents serve
as the vertices of the quiver.
Pfeiffer's article \cite{gotz} 
builds on the idempotent construction above
and shows how one can construct the quiver and the relations for the
presentation of $\Sigma\left(W\right)$.
Since Pfeiffer's construction provides the foundation for this article,
we briefly summarize it in the following theorem.

\begin{theorem}\label{GotzSummary}
Let $\left(W,S\right)$ be a finite Coxeter system.
Then there exist
\begin{itemize}
\item a category $\mathcal{A}$
\item an action of the free monoid $S^\ast$ on $\mathcal{A}$
that partitions $\mathcal{A}$ into orbits
\item subsets $\Lambda$ and $\mathcal{E}$ of the set $\mathcal{X}$
of orbits of $\mathcal{A}$
\item a linear map $\Delta:k\mathcal{A}\to k\mathcal{P}$
(where $\mathcal{P}$ is the power set of $S$)
\end{itemize}
such that
\begin{itemize}
\item $k\mathcal{X}$ is a subalgebra of $k\mathcal{A}$ (where we identify 
the orbit of an element of $\mathcal{A}$
with the sum of its elements in $k\mathcal{A}$)
\item $\Lambda$ is a complete set of pairwise orthogonal primitive idempotents of $k \mathcal{X}$
\item $\lambda\left(k\mathcal{X}\right)\lambda' \cap \mathcal{X}$ is a basis of the subspace 
$\lambda\left(k\mathcal{X}\right)\lambda'$ for all $\lambda, \lambda' \in \Lambda$
\item the pair $\left(Q,\ker\Delta\right)$
is a quiver presentation of $\Sigma\left(W\right)^{\mathrm{op}}$
where $Q$ is the quiver with vertices $\Lambda$ and edges $\mathcal{E}$.
\end{itemize}
\end{theorem}

We briefly repeat the definitions of the constructions introduced
in \autoref{GotzSummary} needed in this article. The category
\[\mathcal{A}=\left\{\left(J;s_1,s_2,\ldots,s_l\right)\mid\rule{0pt}{12pt}
\text{$\left\{s_1,s_2,\ldots,s_l\right\}\subseteq
J\subseteq S$ with $s_1,s_2,\ldots,s_l$ distinct}
\right\}\]
has a partial product
$\circ:\mathcal{A}\times\mathcal{A}\to\mathcal{A}$ defined by
\[\left(J;s_1,s_2,\ldots,s_l\right)\circ\left(K;t_1,t_2,\ldots,t_m\right)
=\left(J;s_1,s_2,\ldots,s_l,t_1,t_2,\ldots,t_m\right)\]
if $K=J\setminus\left\{s_1,s_2,\ldots,s_l\right\}$.
The action of $S^\ast$ on $\mathcal{A}$ is given by
\begin{equation}\label{LSDefinition}
\left(J;s_1,s_2,\ldots,s_l\right).t 
=\left(J^\omega;s_1^\omega,s_2^\omega,\ldots,s_l^\omega\right)
\end{equation}
for $t\in S$ and $\left(J;s_1,s_2,\ldots,s_l\right)\in\mathcal{A}$ where 
$\omega=w_Jw_{J\cup\left\{t\right\}}$ and where $w_J$ and 
$w_{J\cup\left\{t\right\}}$ are the longest elements of the parabolic subgroups
$W_J$ and $W_{J\cup\left\{t\right\}}$ respectively.
The superscripts in \autoref{LSDefinition} denote conjugation,
so for example $s_1^\omega=\omega^{-1}s_1\omega$.

We define a difference operator $\delta$
on $k\mathcal{A}$
by $\delta\left(a\right)=a$ if $l=0$
or by $\delta\left(a\right)=b-b.s_1$ if $l>0$ 
for all $a=\left(J;s_1,s_2,\ldots,s_l\right)$ 
where $b=\left(J\setminus\left\{s_1\right\};s_2,\ldots,s_l\right)$.
Then $\Delta$ is defined by iterating $\delta$ as many times as possible, so 
$\Delta\left(a\right)=\delta^l\left(a\right)$
for $a\in\mathcal{A}$ as above.
Finally, $\Lambda$ is the set of orbits of elements of the form
$\left(J;\right)$ and $\mathcal{E}$ can be calculated
using an algorithm.

Once the quiver provided by \autoref{GotzSummary} has been identified,
it remains to calculate the relations of the presentation.
While difficult in practice, this
amounts in principle only to transferring $\ker\Delta$ to $kQ$.
Pfeiffer \cite{exceptional} has done this 
with his explicit quiver presentations of the
descent algebras of the Coxeter
groups of exceptional and non-crystallographic type.
Other than these calculations, no quiver presentations
of descent algebras are known. However, in contrast with
the finite calculations in \cite{exceptional}, this paper deals with
the calculation of presentations of the algebras in
the infinite family $\left\{\Sigma\left(\sym{n}\right)\mid n\ge 0\right\}$.

The following is an outline of this paper.
The algebras and maps introduced in the outline
are shown in the following diagram.
\[\xymatrix{
kQ_n\ar[r]^\iota&k\mathcal{L}_n\ar[r]\ar[d]^\E&kL_n\ar[d]^\E\ar[dr]^\Delta\\
&k\mathcal{M}_n\ar[r]&kM_n\ar[r]^\pi&k\mathbb{N}^\ast
}\]
To calculate a presentation of $\Sigma\left(\sym{n}\right)$
we first develop a simpler description of $\mathcal{A}$.
Namely, we show in \autoref{EquivalenceSection} that
each element of $\mathcal{A}$ can be represented as a sequence of binary trees,
or a {\em forest}.
The category $L_n$ in the diagram above is the category
of forests corresponding with
elements of $\mathcal{A}$.
The definition and basic properties of forests
are the subject of \autoref{TreeSection}.
We show in \autoref{ActionSection} that
the monoid action of $S^\ast$ on $L_n$
amounts simply to rearrangement
of the trees of a forest, so the $S^\ast$-orbit of
an element of $\mathcal{A}$ corresponds with
the sum of all rearrangements of the corresponding forest.
This action yields a set
$\mathcal{L}_n$ of orbit sums in $kL_n$
corresponding with $\mathcal{X}$ in \autoref{GotzSummary}.
We show in \autoref{DeltaSection} that
the map $\Delta$ also has a simple description
when we represent the elements of $\mathcal{A}$ as forests.
Namely, we introduce sets $M_n$ and $\mathcal{M}_n$
analogous to $L_n$ and $\mathcal{L}_n$
in \autoref{UnlabeledSection} 
and we show in \autoref{AlignmentSection}
that $\Delta$ factors through $kM_n$
as the composition of a natural map $\E:kL_n\to kM_n$
with the map $\pi:kM_n\to k\mathbb{N}^\ast$ that
replaces all the nodes of a forest with the Lie 
bracket in the free associative algebra $k\mathbb{N}^\ast$.
This allows us to identify $\Sigma\left(\sym{n}\right)^\mathsf{op}$
with a quotient of $k\mathcal{L}_n$ in \autoref{LIsomorphism}.
We introduce a quiver $Q_n$ in \autoref{QuiverSection}
and show in \autoref{BranchSection}
that the path algebra of $Q_n$ can be embedded into
the algebra $k\mathcal{L}_n$ of forest classes
through the injective anti-homomorphism $\iota$ shown in the diagram above.
We also show in \autoref{FactorizationSection}
that $Q_n$ is the ordinary quiver of $\Sigma\left(\sym{n}\right)$.
This means that $\Sigma\left(\sym{n}\right)$
can be identified with the quotient of
the path algebra of $Q_n$ by an ideal that can be explicitly calculated.
We conjecture in \autoref{RelationsSection} that
a generating set for this ideal can be produced through a simple procedure.
Finally we calculate the presentation of $\Sigma\left(\sym{8}\right)$
in \autoref{ExampleSection}, thus verifying our conjecture
in this particular example. 

\section{Compositions, partitions, and rearrangement}\label{CompositionSection}
Much of the charm of the theory developed in this paper
stems from the reduction of complicated combinatorial operations
to the simpler operation of rearrangement,
which is the subject of this section.
We denote the free monoid on a set $\Omega$ by $\Omega^\ast$.
This is the set of all formal products $x_1x_2\cdots x_j$ where
$x_i\in\Omega$ for all $1\le i\le j$.
The binary operation on $\Omega^\ast$ is not denoted.
In this paper, an important instance
of this construction occurs when $\Omega$
is the set $\mathbb{N}$ of natural numbers, which does {\em not} include $0$.
The elements of $\mathbb{N}^\ast$ are called {\em compositions}
and the numbers $x_i$ in a composition $x_1x_2\cdots x_j$
are called its {\em parts}.

The symmetric group $\sym{j}$ acts on compositions with $j$ parts by
\[\left(x_1x_2\cdots x_j\right).\pi
=x_{1.\pi^{-1}}x_{2.\pi^{-1}}\cdots x_{j.\pi^{-1}}\]
for $\pi\in\sym{j}$. This action 
is called the {\em P\'olya action}.
The orbits of the P\'olya action on $\mathbb{N}^\ast$
are called {\em partitions}. We represent 
a partition by any of its representatives
when this causes no confusion.

\section{Trees and forests}\label{TreeSection}
A {\em (binary) tree} is either a natural number or a diagram
$\vcenter{\begin{xy}<.2cm,0cm>:
(0,1)="1";
"1";(-1,0)*+!U{X}**\dir{-};
"1";(1,0)*+!U{Y}**\dir{-};
\end{xy}}$ where $X$ and $Y$ are trees.
Trees of the first type are called {\em leaves} while
trees of the second type are called {\em nodes}.
A {\em labeled forest} is a sequence of trees
whose nodes are labeled by natural numbers in such a way that the label 
of every node is greater than that of its parent if it has one,
and each number $1,2,\ldots,l$ is the label of exactly one node,
where $l$ is the number of nodes in the sequence.
For example
\begin{equation}\label{YForest}
\vcenter{\begin{xy}<.3cm,0cm>:
(2,3)="1"*+!U{_3};
"1";(1,2)*+!U{\color{red}1}**\dir{-};
"1";(3,2)*+!U{\color{red}2}**\dir{-};
(5,3)="1"*+!U{_1};
"1";(4,2)*+!U{\color{red}1}**\dir{-};
"1";(7,2)="2"*+!U{_2}**\dir{-};
"2";(6,1)*+!U{\color{red}3}**\dir{-};
"2";(8,1)*+!U{\color{red}1}**\dir{-};
(10,3)="1"*+!U{_4};
"1";(9,2)*+!U{\color{red}2}**\dir{-};
"1";(11,2)*+!U{\color{red}1}**\dir{-};
\end{xy}}
\end{equation}
is a labeled forest. Let $Y$ be a labeled forest.
The sequence of leaves of $Y$
is called its {\em foliage} and is denoted by $\f{Y}$.
The sum of the leaves of a tree is called its {\em value}.
The sequence of values of the trees of $Y$
is called its {\em squash} and is denoted by $\s{Y}$.
The number of nodes in $Y$
is called its {\em length}
and is denoted by $\ell\left(Y\right)$.
For example, if $Y$ is the forest shown in \autoref{YForest}
then $\f{Y}=1213121$ and $\s{Y}=353$
while $\ell\left(Y\right)=4$.

Whenever two forests $X$ and $Y$ satisfy $\f{X}=\s{Y}$
we define a product $X\bullet Y$ by replacing the leaves
of $X$ with the trees of $Y$.
For example, if $X$ is the forest
$\vcenter{\begin{xy}<.25cm,0cm>:
(2,2)="1"*+!U{_1};
"1";(1,1)*+!U{\color{red}3}**\dir{-};
"1";(3,1)*+!U{\color{red}5}**\dir{-};
(4,2)*+!U{\color{red}3};
\end{xy}}$ and $Y$ is the forest shown in \autoref{YForest}
then $\f{X}=353=\s{Y}$ so that
\begin{equation}\label{XYForest}
\vcenter{\begin{xy}<.3cm,0cm>:
(4,4)="1"*+!U{_1};
"1";(2,3)="2"*+!U{_4}**\dir{-};
"2";(1,2)*+!U{\color{red}1}**\dir{-};
"2";(3,2)*+!U{\color{red}2}**\dir{-};
"1";(6,3)="3"*+!U{_2}**\dir{-};
"3";(5,2)*+!U{\color{red}1}**\dir{-};
"3";(8,2)="4"*+!U{_3}**\dir{-};
"4";(7,1)*+!U{\color{red}3}**\dir{-};
"4";(9,1)*+!U{\color{red}1}**\dir{-};
(11,4)="1"*+!U{_5};
"1";(10,3)*+!U{\color{red}2}**\dir{-};
"1";(12,3)*+!U{\color{red}1}**\dir{-};
\end{xy}}
\end{equation}
is the product $X\bullet Y$.
Note that the node labels of $Y$ must be incremented 
by $\ell\left(X\right)$ to ensure that
the product will also be a labeled forest.

All the definitions above can be made mathematically precise
by defining the set of labeled trees $\mathbb{L}$
to be the minimal set containing
$\mathbb{N}$ and also containing the tuple
$\left(X_1,i,X_2\right)$ whenever
$X_1,X_2\in\mathbb{L}$ and $i\in\mathbb{N}$.
A labeled tree of the form $\left(X_1,i,X_2\right)$
should also satisfy the labeling condition
$i_1,i_2<i$ where $X_1=\left(X_{11},i_1,X_{12}\right)$
and $X_2=\left(X_{21},i_2,X_{22}\right)$.
We define the squash of a labeled tree $X\in\mathbb{L}$ by the formula
\begin{equation}\label{SquashTreeDefinition}
\s{X}=\begin{cases}X&\text{if $X\in\mathbb{N}$}\\
\s{X_1}+\s{X_2}&\text{if $X=\left(X_1,i,X_2\right)$}\end{cases}
\end{equation}
and we similarly define the foliage and length of $X$.
Then a labeled forest is an element of the
free monoid on $\mathbb{L}$ which satisfies the labeling condition
dealing with unique node labels from the original definition.
The definition in \autoref{SquashTreeDefinition}
extends to labeled forests by
$\s{X_1X_2\cdots X_j}=
\s{X_1}\;
\s{X_2}\cdots
\s{X_j}$ where $X_1,X_2,\ldots,X_j\in\mathbb{L}$
and similarly for the foliage and length analogs of
\autoref{SquashTreeDefinition}.

\begin{lemma}\label{UniqueFactorization}
A labeled forest of length at least one
can be uniquely factorized as a product of labeled forests
of length one.
\end{lemma}

\begin{proof}
Suppose that $X=X_1X_2\cdots X_j$ is a labeled
forest, where $X_1,X_2,\ldots,X_j$ are trees.
Note that since $1$ is the smallest node label of $X$,
it must be the label of one of the trees $X_1,X_2,\ldots,X_j$, say $X_i$.
This means that
$X_i=\vcenter{\begin{xy}<.3cm,0cm>:
(2,2)="1"*+!U{_1};
"1";(1,1)*+!U{X_{i1}}**\dir{-};
"1";(3,1)*+!U{X_{i2}}**\dir{-};
\end{xy}}$ for some trees $X_{i1}$ and $X_{i2}$.
Let $Y$ be the forest obtained from
$X_1X_2\cdots X_{i-1}X_{i1}X_{i2}X_{i+1}\cdots X_j$
by reducing the node labels by one
and write $x_1x_2\cdots x_{i-1}x_{i1}x_{i2}x_{i+1}\cdots x_j=\s{Y}$. We put
\[X'=\vcenter{\begin{xy}<.3cm,0cm>:
(0,2)*+!UR{{\color{red}x_1x_2}\cdots{\color{red}x_{i-1}}};
(2,2)="1"*+!U{_1};
"1";(1,1)*+!U{{\color{red}x_{i1}}}**\dir{-};
"1";(3,1)*+!U{{\color{red}x_{i2}}}**\dir{-};
(4,2)*+!UL{{\color{red}x_{i+1}}\cdots{\color{red}x_j}};
\end{xy}}\]
so that $X=X'\bullet Y$.
Note that $X'$ is the unique forest of length one with squash $\s{X}$
and foliage $\s{Y}$.
Repeating the procedure with $Y$ in place
of $X$ yields the desired factorization by induction.
\end{proof}

For example, the forest in \autoref{XYForest}
can be factorized as
\begin{equation}
\label{XYFactorization}
\left(\vcenter{\begin{xy}<.2cm,0cm>:
(2,2)="1"*+!U{_1};
"1";(1,1)*+!U{\color{red}3}**\dir{-};
"1";(3,1)*+!U{\color{red}5}**\dir{-};
(4,2)*+!UL{\color{red}3};
\end{xy}}\right)\bullet
\left(\vcenter{\begin{xy}<.2cm,0cm>:
(1,2)*+!UR{\color{red}3};
(3,2)="1"*+!U{_1};
"1";(2,1)*+!U{\color{red}1}**\dir{-};
"1";(4,1)*+!U{\color{red}4}**\dir{-};
(5,2)*+!UL{\color{red}3};
\end{xy}}\right)\bullet\left(
\vcenter{\begin{xy}<.2cm,0cm>:
(1,2)*+!UR{\color{red}31};
(3,2)="1"*+!U{_1};
"1";(2,1)*+!U{\color{red}3}**\dir{-};
"1";(4,1)*+!U{\color{red}1}**\dir{-};
(5,2)*+!UL{\color{red}3};
\end{xy}}
\right)\bullet\left(
\vcenter{\begin{xy}<.2cm,0cm>:
(2,2)="1"*+!U{_1};
"1";(1,1)*+!U{\color{red}1}**\dir{-};
"1";(3,1)*+!U{\color{red}2}**\dir{-};
(4,2)*+!UL{\color{red}1313};
\end{xy}}
\right)\bullet\left(
\vcenter{\begin{xy}<.2cm,0cm>:
(1,2)*+!UR{\color{red}12131};
(3,2)="1"*+!U{_1};
"1";(2,1)*+!U{\color{red}2}**\dir{-};
"1";(4,1)*+!U{\color{red}1}**\dir{-};
\end{xy}}\right).
\end{equation}

The {\em value} of a forest is the sum of the values of its trees.
For the purpose of constructing a quiver presentation of
$\Sigma\left(\sym{n}\right)$ we restrict our attention to
the set $L_n$ of forests of value
$n\in\mathbb{N}$.
Then $L_n$ is a {\em category}, that is, a monoid whose product is only
partially defined.
Taking $X\bullet Y$ to be zero whenever
$\f{X}\ne\s{Y}$ makes $kL_n$ into a $k$-algebra.

\section{Equivalence of forests with alleys}\label{EquivalenceSection}
Recall from \autoref{IntroductionSection} that
\[\mathcal{A}=\left\{\left(J;s_1,s_2,\ldots,s_l\right)\mid\rule{0pt}{12pt}
\text{$\left\{s_1,s_2,\ldots,s_l\right\}\subseteq
J\subseteq S$ with $s_1,s_2,\ldots,s_l$ distinct}
\right\}\]
and that the partial product
$\circ:\mathcal{A}\times\mathcal{A}\to\mathcal{A}$ is defined by
\[\left(J;s_1,s_2,\ldots,s_l\right)\circ\left(K;t_1,t_2,\ldots,t_m\right)
=\left(J;s_1,s_2,\ldots,s_l,t_1,t_2,\ldots,t_m\right)\]
if $K=J\setminus\left\{s_1,s_2,\ldots,s_l\right\}$.
The category $\mathcal{A}$ is a combinatorial gadget used to construct
quiver presentations of the descent algebras of finite
Coxeter groups.
The elements of $\mathcal{A}$ are called {\em alleys}.
The number~$l$ is called the {\em length}
of the alley $a=\left(J;s_1,s_2,\ldots,s_l\right)$
and is denoted by $\ell\left(a\right)$.
One can also view $a$ as the chain
\begin{equation}\label{Chain}
J\supseteq J\setminus\left\{s_1\right\}\supseteq
J\setminus\left\{s_1,s_2\right\}\supseteq\cdots\supseteq
J\setminus\left\{s_1,s_2,\ldots,s_l\right\}
\end{equation}
of subsets of $\left\{1,2,\ldots,n\right\}$.
Then the product of two alleys corresponds
with the concatenation
of the corresponding chains whenever the concatenation is also a chain.

\begin{proposition}\label{Equivalence}
The category $\mathcal{A}$ 
associated to the Coxeter group $\sym{n}$
is equivalent to $L_n$
through a length-preserving functor.
\end{proposition}

\begin{proof}
We identify the Coxeter generating set $S$ of $\sym{n}$ 
with the set $\left\{1,2,\ldots,n-1\right\}$.
If $J\subseteq S$ with $\left|J\right|=n-j$ then we
write $S\setminus J
=\left\{t_1,t_2,\ldots,t_{j-1}\right\}$
where $t_1<t_2<\cdots<t_{j-1}$.
We put $t_0=0$ and $t_j=n$ and let
$\varphi\left(J\right)$ be the composition $q_1q_2\cdots q_j$
where $q_i=t_i-t_{i-1}$.
Then $\varphi$ is a bijection between the subsets of
$S$ and the compositions of $n$.

Let $H_m$ be the Hasse diagram of the relation $\subseteq$
on the subsets of $\left\{1,2,\ldots,m\right\}$ for
$m\ge 0$. Then $H_m$
is a quiver with a vertex for every subset of $\left\{1,2,\ldots,m\right\}$
and an edge from $J$ to $K$ if $J\subseteq K$ and
$\left|K\setminus J\right|=1$.
Thanks to the description in \autoref{Chain}
we can identify $\mathcal{A}$ with the set of paths in $H_{n-1}$.
Note that under this identification the length of an alley
equals the length of the corresponding path.

Now consider the quiver $H_n'$ which has a vertex for every composition
of $n$ and an edge from $p$ to $q$ if there exists a forest of length one
with foliage $p$ and squash $q$. 
Thanks to \autoref{UniqueFactorization} we can identify $L_n$ with the set
of paths in $H_n'$.
Note that under this identification the length of a forest
equals the length of the corresponding path.

Next we observe that the vertices of $H_{n-1}$ are in bijection
with the vertices of $H_n'$ through $\varphi$ and that $H_{n-1}$
has an edge from
$J$ to $K$ if and only if $H_n'$ has an edge from $\varphi\left(J\right)$
to $\varphi\left(K\right)$. This means that the quivers
$H_{n-1}$ and $H_n'$ are isomorphic as directed graphs so that
$\mathcal{A}$ and $L_n$ are equivalent through a length-preserving
functor, which we denote by $\varphi$ in the following sections.
\end{proof}

For example, the alley
$\left(\left\{1,2,3,4,5,6,7,9,10\right\};3,4,7,1,10\right)$
corresponds with the path
\begin{multline*}
\left\{2,5,6,9\right\}
\to\left\{2,5,6,9,10\right\}
\to\left\{1,2,5,6,9,10\right\}
\to\left\{1,2,5,6,7,9,10\right\}\\
\to\left\{1,2,4,5,6,7,9,10\right\}
\to\left\{1,2,3,4,5,6,7,9,10\right\}
\end{multline*}
in $H_{10}$, which in turn corresponds under $\varphi$ with the path
\[1213121\to 121313\to 31313\to 3143\to 353\to 83\]
in $H'_{11}$ corresponding with the forest shown in \autoref{XYForest}
and factorized in \autoref{XYFactorization}.

\section{Actions and orbits}\label{ActionSection}
If $X=X_1X_2\cdots X_j$ is a labeled forest where $X_1,X_2,\ldots,X_j$
are trees, then the trees $X_1,X_2,\ldots,X_j$
are called the {\em parts} of $X$.
The P\'olya action of $\sym{j}$ on compositions with $j$~parts
extends to an action on forests with $j$~parts.
If $X$ is a forest with $j$~parts,
then we denote the sum of the elements
in the same $\sym{j}$-orbit as $X$ by $\left[X\right]$.
For example, if $X$ is the forest
$\vcenter{\begin{xy}<.2cm,0cm>:
(2,3)="1"*+!U{_1};
"1";(1,2)*+!U{\color{red}1}**\dir{-};
"1";(3,2)*+!U{\color{red}2}**\dir{-};
(5,3)="1"*+!U{_2};
"1";(4,2)*+!U{\color{red}1}**\dir{-};
"1";(6,2)*+!U{\color{red}2}**\dir{-};
(8,3)="1"*+!U{_3};
"1";(7,2)*+!U{\color{red}1}**\dir{-};
"1";(10,2)="2"*+!U{_4}**\dir{-};
"2";(9,1)*+!U{\color{red}1}**\dir{-};
"2";(11,1)*+!U{\color{red}2}**\dir{-};
\end{xy}}$ then
\begin{align*}
\left[X\right]
=&\vcenter{\begin{xy}<.3cm,0cm>:
(2,3)="1"*+!U{_1};
"1";(1,2)*+!U{\color{red}1}**\dir{-};
"1";(3,2)*+!U{\color{red}2}**\dir{-};
(5,3)="1"*+!U{_2};
"1";(4,2)*+!U{\color{red}1}**\dir{-};
"1";(6,2)*+!U{\color{red}2}**\dir{-};
(8,3)="1"*+!U{_3};
"1";(7,2)*+!U{\color{red}1}**\dir{-};
"1";(10,2)="2"*+!U{_4}**\dir{-};
"2";(9,1)*+!U{\color{red}1}**\dir{-};
"2";(11,1)*+!U{\color{red}2}**\dir{-};
\end{xy}}
+\vcenter{\begin{xy}<.3cm,0cm>:
(2,3)="1"*+!U{_2};
"1";(1,2)*+!U{\color{red}1}**\dir{-};
"1";(3,2)*+!U{\color{red}2}**\dir{-};
(5,3)="1"*+!U{_1};
"1";(4,2)*+!U{\color{red}1}**\dir{-};
"1";(6,2)*+!U{\color{red}2}**\dir{-};
(8,3)="1"*+!U{_3};
"1";(7,2)*+!U{\color{red}1}**\dir{-};
"1";(10,2)="2"*+!U{_4}**\dir{-};
"2";(9,1)*+!U{\color{red}1}**\dir{-};
"2";(11,1)*+!U{\color{red}2}**\dir{-};
\end{xy}}
+\vcenter{\begin{xy}<.3cm,0cm>:
(2,3)="1"*+!U{_1};
"1";(1,2)*+!U{\color{red}1}**\dir{-};
"1";(3,2)*+!U{\color{red}2}**\dir{-};
(5,3)="1"*+!U{_3};
"1";(4,2)*+!U{\color{red}1}**\dir{-};
"1";(7,2)="2"*+!U{_4}**\dir{-};
"2";(6,1)*+!U{\color{red}1}**\dir{-};
"2";(8,1)*+!U{\color{red}2}**\dir{-};
(10,3)="1"*+!U{_2};
"1";(9,2)*+!U{\color{red}1}**\dir{-};
"1";(11,2)*+!U{\color{red}2}**\dir{-};
\end{xy}}\\
+&\vcenter{\begin{xy}<.3cm,0cm>:
(2,3)="1"*+!U{_2};
"1";(1,2)*+!U{\color{red}1}**\dir{-};
"1";(3,2)*+!U{\color{red}2}**\dir{-};
(5,3)="1"*+!U{_3};
"1";(4,2)*+!U{\color{red}1}**\dir{-};
"1";(7,2)="2"*+!U{_4}**\dir{-};
"2";(6,1)*+!U{\color{red}1}**\dir{-};
"2";(8,1)*+!U{\color{red}2}**\dir{-};
(10,3)="1"*+!U{_1};
"1";(9,2)*+!U{\color{red}1}**\dir{-};
"1";(11,2)*+!U{\color{red}2}**\dir{-};
\end{xy}}
+\vcenter{\begin{xy}<.3cm,0cm>:
(2,3)="1"*+!U{_3};
"1";(1,2)*+!U{\color{red}1}**\dir{-};
"1";(4,2)="2"*+!U{_4}**\dir{-};
"2";(3,1)*+!U{\color{red}1}**\dir{-};
"2";(5,1)*+!U{\color{red}2}**\dir{-};
(7,3)="1"*+!U{_1};
"1";(6,2)*+!U{\color{red}1}**\dir{-};
"1";(8,2)*+!U{\color{red}2}**\dir{-};
(10,3)="1"*+!U{_2};
"1";(9,2)*+!U{\color{red}1}**\dir{-};
"1";(11,2)*+!U{\color{red}2}**\dir{-};
\end{xy}}
+\vcenter{\begin{xy}<.3cm,0cm>:
(2,3)="1"*+!U{_3};
"1";(1,2)*+!U{\color{red}1}**\dir{-};
"1";(4,2)="2"*+!U{_4}**\dir{-};
"2";(3,1)*+!U{\color{red}1}**\dir{-};
"2";(5,1)*+!U{\color{red}2}**\dir{-};
(7,3)="1"*+!U{_2};
"1";(6,2)*+!U{\color{red}1}**\dir{-};
"1";(8,2)*+!U{\color{red}2}**\dir{-};
(10,3)="1"*+!U{_1};
"1";(9,2)*+!U{\color{red}1}**\dir{-};
"1";(11,2)*+!U{\color{red}2}**\dir{-};
\end{xy}}.
\end{align*}
The set of orbit sums in $kL_n$ is denoted by $\mathcal{L}_n$.

Suppose that $X,Y\in L_n$ are such that $\underline{X}=\overline{Y}$.
If $X$ has $i$~parts and $Y$ has $j$~parts,
then any element $\sigma\in\sym{i}$ induces
a permutation $\tau\in\sym{j}$ of the leaves of $X$. Namely, $\tau$
is the permutation satisfying
$X.\sigma\bullet Y.\tau=\left(X\bullet Y\right).\sigma$.
This correspondence is an injective homomorphism when restricted to
any subgroup of $\sym{i}$ that permutes only parts of $X$ that
have the same number of leaves.
The stabilizer of $X$ in $\sym{i}$ is such a subgroup,
since it permutes only parts of length zero,
the parts of positive length having distinct node labels.
Therefore the stabilizer of $X$ is isomorphic to a subgroup $K$ of $\sym{j}$.
Now if $H$ is the stabilizer of $Y$ in $\sym{j}$ then
\[\left[X\right]\bullet\left[Y\right]
=\sum_{t=1}^m\left[X\bullet Y.\sigma_t\right]\in k\mathcal{L}_n\]
where $\sigma_1,\sigma_2,\ldots,\sigma_m$ are representatives
of the double cosets of $H,K$ in $\sym{j}$.
This proves the following proposition.
\begin{proposition}\label{LSubalgebra}
$k\mathcal{L}_n$ is a subalgebra of $kL_n$.
\end{proposition}
Alternately, \autoref{LSubalgebra} follows with \autoref{GotzSummary} from
\autoref{OrbitsCorrespond} below
through the equivalence of $L_n$ with $\mathcal{A}$.

Recall from \autoref{IntroductionSection} that
the free monoid $S^\ast$ acts on $\mathcal{A}$ by
\begin{equation*}
\left(J;s_1,s_2,\ldots,s_l\right).t 
=\left(J^\omega;s_1^\omega,s_2^\omega,\ldots,s_l^\omega\right)
\end{equation*}
for $t\in S$ and $\left(J;s_1,s_2,\ldots,s_l\right)\in\mathcal{A}$ where 
$\omega=w_Jw_{J\cup\left\{t\right\}}$ and where $w_J$ and 
$w_{J\cup\left\{t\right\}}$ are the longest elements of the parabolic subgroups
$W_J$ and $W_{J\cup\left\{t\right\}}$ respectively.
When $\mathcal{A}$ is the category associated with $\Sigma\left(\sym{n}\right)$
we calculate the orbits of this action in the following proposition.

\begin{proposition}\label{OrbitsCorrespond}
The orbits of the P\'olya action on $L_n$
correspond under the equivalence $\varphi$ in 
\autoref{Equivalence} with the $S^\ast$-orbits on $\mathcal{A}$.
\end{proposition}

\begin{proof}
Let $a=\left(J;s_1,s_2,\ldots,s_l\right)\in\mathcal{A}$ and let
$X=\varphi\left(a\right)\in L_n$.
Let $t_0,t_1,\ldots,t_j$ be as in the proof of \autoref{Equivalence}.
Note that if $t\in J$ then $\omega=w_Jw_{J\cup\left\{t\right\}}=\mathsf{id}_W$
so that $a.t=a$.
Otherwise assume that $t=t_i$ for some $1\le i\le j-1$.
We claim that $\varphi\left(a.t_i\right)$
is obtained from $X$ by exchanging the parts 
in positions $i$ and $i+1$.
From this it will follow that
$\varphi\left(a.S^\ast\right)
=\varphi\left(a\right).\sym{j}$.

It is easy to see that conjugation by $w_J$ reverses the elements in the block 
\[B_g=\left\{t_g+1,t_g+2,\ldots,t_{g+1}-1\right\}\]
for all $0\le g\le j-1$.
Note that including $t_i$ in $J$ joins the blocks $B_{i-1}$ and $B_i$
into the block $B_{i-1}\cup\left\{t_i\right\}\cup B_i$.
Then since conjugation by $w_{J\cup\left\{t_i\right\}}$
again reverses all the blocks,
the effect of conjugation by $\omega$
is to shift $B_{i-1}$ to the right of $B_i$
while fixing the remaining blocks.

It follows from the definition of $\varphi$
that if $K\subseteq J$ then $\varphi\left(K\right)$
is a refinement of $\varphi\left(J\right)$. In other words,
if $\varphi\left(J\right)=q_1q_2\cdots q_j$ where
$q_1,q_2,\ldots,q_j\in\mathbb{N}$ then
$\varphi\left(K\right)=p_1p_2\cdots p_j$ where $p_i$
is a composition of $q_i$ for all $1\le i\le j$.
Then conjugating $K$ by $\omega$ corresponds under $\varphi$
with exchanging the compositions $p_i$ and $p_{i+1}$ of $\varphi\left(K\right)$.
Applying this observation to each vertex $K$ of the path
corresponding with $a$ in $H_{n-1}$ we see that
the path in $H_n'$ corresponding with $\varphi\left(a.t_i\right)$
is obtained from the path corresponding with
$X=\varphi\left(a\right)$ by exchanging the compositions
$p_i$ and $p_{i+1}$ of each vertex $\varphi\left(K\right)$. Therefore
$\varphi\left(a.t_i\right)$ is obtained from $X$ by exchanging
the parts in positions $i$ and $i+1$.
\end{proof}

\section{Difference operators}\label{DeltaSection}
In this section we prove 
one of the main results of this paper, namely that
$\Sigma\left(\sym{n}\right)$ is isomorphic to
a quotient of $k\mathcal{L}_n$.
For this purpose we define a difference operator $\delta$ on $kL_n$ as follows.
If $\ell\left(X\right)=0$ then we define $\delta\left(X\right)=X$.
Otherwise suppose that $X=X_1X_2\cdots X_j\in L_n$ where $X_1,X_2,\ldots,X_j$
are trees and that $X_i$ is the node of $X$ labeled $1$.
Then $X_i=\vcenter{\begin{xy}<.3cm,0cm>:
(2,2)="1"*+!U{_1};
"1";(1,1)*+!U{X_{i1}}**\dir{-};
"1";(3,1)*+!U{X_{i2}}**\dir{-};
\end{xy}}$ for some trees $X_{i1}$ and $X_{i2}$.
We define $\delta\left(X\right)$ to be the
element of $kL_n$ obtained from $X$ by replacing $X_i$
with the Lie bracket $X_{i1}X_{i2}-X_{i2}X_{i1}$
and reducing the remaining node labels by one.
In terms of the P\'olya action,
this means that $\delta\left(X\right)=Y-Y.i$
where $Y$ is the forest obtained from $X$ by
splitting the part
$\vcenter{\begin{xy}<.3cm,0cm>:
(2,2)="1"*+!U{_1};
"1";(1,1)*+!U{X_{i1}}**\dir{-};
"1";(3,1)*+!U{X_{i2}}**\dir{-};
\end{xy}}$ in position $i$
into $X_{i1}X_{i2}$
and reducing the remaining node labels by one.

Recall from \autoref{IntroductionSection} that
the difference operator $\delta$
on $k\mathcal{A}$ is defined
by $\delta\left(a\right)=a$ if $l=0$
or by $\delta\left(a\right)=b-b.s_1$ if $l>0$ 
for all $a=\left(J;s_1,s_2,\ldots,s_l\right)$ 
where $b=\left(J\setminus\left\{s_1\right\};s_2,\ldots,s_l\right)$.
When $\mathcal{A}$ is the category
associated with $\Sigma\left(\sym{n}\right)$
this difference operator coincides with the one introduced above
in the following sense.

\begin{proposition}\label{DeltaCoincides}
$\varphi\left(\delta\left(a\right)\right)
=\delta\left(\varphi\left(a\right)\right)$
for all alleys $a\in\mathcal{A}$.
\end{proposition}

\begin{proof} 
Let $a$ and $b$ be as above
and let $X=X_1X_2\cdots X_j=\varphi\left(a\right)\in L_n$
where $X_1,X_2,\ldots,X_j$ are trees.
The factorization
$a=\left(J;s_1\right)\circ b$
and the factorization $X=X'\bullet Y$ in \autoref{UniqueFactorization}
imply that $\varphi\left(J;s_1\right)=X'$ and $\varphi\left(b\right)=Y$
by unique factorization and length-preserving equivalence.

Now let $t_1,\ldots,t_{j-1}$ be as in \autoref{Equivalence}. Then
\[\left\{1,2,\ldots,n-1\right\}\setminus
\left(J\setminus\left\{s_1\right\}\right)
=\left\{t_1,t_2,\ldots,t_{i-1},s_1,t_i,t_{i+1},\ldots t_{j-1}\right\}\]
with $t_1<t_2<\cdots<t_{i-1}<s_1<t_i<t_{i+1}<\cdots<t_{j-1}$.
Since $s_1$ is in position~$i$ of this list,
$\varphi\left(b.s_1\right)$ is obtained from
$\varphi\left(b\right)$ by exchanging the trees
in positions $i$ and $i+1$ by the proof of \autoref{OrbitsCorrespond}.
Thus $\delta\left(X\right)=Y-Y.i=
\varphi\left(b-b.s_1\right)
=\varphi\left(\delta\left(a\right)\right)$.
\end{proof}

Iterating $\delta$ as many times as possible determines another
difference operator $\Delta$ on $kL_n$ defined by
$\Delta\left(X\right)=\delta^{\ell\left(X\right)}\left(X\right)$
for all forests $X$. This is analogous to the difference operator
$\Delta$ on $k\mathcal{A}$ defined by
$\Delta\left(a\right)=\delta^{\ell\left(a\right)}\left(a\right)$
for all alleys $a$.
Note that applying $\Delta$ to $X\in L_n$
results in a $\mathbb{Z}$-linear combination of compositions of $n$.

\begin{theorem}\label{LIsomorphism}
$\Sigma\left(\sym{n}\right)^\mathsf{op}$ is isomorphic to
$k\mathcal{L}_n/\ker\Delta$.
\end{theorem}

\begin{proof}
$k\mathcal{X}/\ker\Delta$ is isomorphic to
$\Sigma\left(\sym{n}\right)^\mathsf{op}$
by \autoref{GotzSummary} and
$k\mathcal{X}$ is isomorphic to $k\mathcal{L}_n$ by
\autoref{OrbitsCorrespond}. Then $k\mathcal{X}/\ker\Delta$
is isomorphic to $k\mathcal{L}_n/\ker\Delta$
since the maps $\Delta$ on the two algebras coincide under $\varphi$
by \autoref{DeltaCoincides}.
\end{proof}

\autoref{LIsomorphism} gives a new construction of
$\Sigma\left(\sym{n}\right)$ as a quotient of $k\mathcal{L}_n$.
We show in the following sections that $k\mathcal{L}_n$ in turn
is a homomorphic image of the path algebra of a quiver.

\section{The quiver}\label{QuiverSection}
Recall from \autoref{UniqueFactorization}
that a labeled forest of length at least one can be uniquely factorized as a
product of forests of length one.
This property fails when we replace
$L_n$ with $\mathcal{L}_n$.
For example, if we try to factorize
$\left[\vcenter{\begin{xy}<.25cm,0cm>:
(2,3)="1"*+!U{_1};
"1";(1,2)*+!U{\color{red}1}**\dir{-};
"1";(4,2)="2"*+!U{_2}**\dir{-};
"2";(3,1)*+!U{\color{red}1}**\dir{-};
"2";(5,1)*+!U{\color{red}2}**\dir{-};
(6,3)*+!U{\color{red}3};
\end{xy}}\right]$
as the product of
$\left[\vcenter{\begin{xy}<.25cm,0cm>:
(2,2)="1"*+!U{_1};
"1";(1,1)*+!U{\color{red}1}**\dir{-};
"1";(3,1)*+!U{\color{red}3}**\dir{-};
(4,2)*+!U{\color{red}3};
\end{xy}}\right]$ and
$\left[\vcenter{\begin{xy}<.25cm,0cm>:
(2,2)="1"*+!U{_1};
"1";(1,1)*+!U{\color{red}1}**\dir{-};
"1";(3,1)*+!U{\color{red}2}**\dir{-};
(4,2)*+!U{\color{red}12};
\end{xy}}\right]$
we find that the product
\[\left[\vcenter{\begin{xy}<.3cm,0cm>:
(2,2)="1"*+!U{_1};
"1";(1,1)*+!U{\color{red}1}**\dir{-};
"1";(3,1)*+!U{\color{red}3}**\dir{-};
(4,2)*+!U{\color{red}3};
\end{xy}}\right]\bullet
\left[\vcenter{\begin{xy}<.3cm,0cm>:
(2,2)="1"*+!U{_1};
"1";(1,1)*+!U{\color{red}1}**\dir{-};
"1";(3,1)*+!U{\color{red}2}**\dir{-};
(4,2)*+!U{\color{red}12};
\end{xy}}\right]
=\left[\vcenter{\begin{xy}<.3cm,0cm>:
(2,3)="1"*+!U{_1};
"1";(1,2)*+!U{\color{red}1}**\dir{-};
"1";(4,2)="2"*+!U{_2}**\dir{-};
"2";(3,1)*+!U{\color{red}1}**\dir{-};
"2";(5,1)*+!U{\color{red}2}**\dir{-};
(6,3)*+!U{\color{red}3};
\end{xy}}\right]
+\left[\vcenter{\begin{xy}<.3cm,0cm>:
(2,2)="1"*+!U{_1};
"1";(1,1)*+!U{\color{red}1}**\dir{-};
"1";(3,1)*+!U{\color{red}3}**\dir{-};
(5,2)="1"*+!U{_2};
"1";(4,1)*+!U{\color{red}1}**\dir{-};
"1";(6,1)*+!U{\color{red}2}**\dir{-};
\end{xy}}\right]\]
has an extra term. This defect in factorization
is the subject of \autoref{FactorizationSection}.

Nonetheless, the success of factorization in $L_n$
suggests representing the algebra
$k\mathcal{L}_n$ as a path algebra.
Namely, in the factorization of any labeled forest,
the foliage of each factor
equals the squash of the following factor, so
we can regard each factor
as an edge from its foliage to its squash.

Let $Q_n$ be the quiver having the partitions of $n$
as vertices and an edge from the vertex $p$ to the vertex $q$
whenever $q$ can be obtained from $p$ by replacing two {\em distinct} parts
with their sum. In other words, $Q_n$ is the Hasse diagram
of the partitions of $n$ under restricted partition refinement,
in contrast with the Hasse diagram $H'_n$
of the {\em compositions} of $n$ under {\em ordinary} partition refinement
introduced in the proof of \autoref{Equivalence}.
For example, the quiver $Q_8$
is shown in \autoref{Q8}, omitting
the vertices $11111111$ and $2222$, which are
not incident with any edges.
This quiver also appears in \cite[p. 54]{bauer}.

\begin{figure}
\caption{The quiver $Q_8$}\label{Q8}
\begin{tikzpicture}[>=latex,join=bevel,scale=.60]
  \pgfsetlinewidth{1bp}
  \pgfsetcolor{black}
  \draw [->] (264bp,74bp) .. controls (278bp,66bp) and (298bp,53bp)  .. (324bp,37bp);
  \draw (295bp,61bp) node {$_{e}$};
  \draw [->] (369bp,234bp) .. controls (380bp,234bp) and (393bp,234bp)  .. (416bp,234bp);
  \draw (392bp,240bp) node {$_{h}$};
  \draw [->] (271bp,18bp) .. controls (280bp,18bp) and (289bp,19bp)  .. (298bp,19bp) .. controls (301bp,19bp) and (303bp,20bp)  .. (316bp,21bp);
  \draw (295bp,25bp) node {$_{f}$};
  \draw [->] (172bp,86bp) .. controls (183bp,86bp) and (196bp,86bp)  .. (218bp,86bp);
  \draw (195bp,91bp) node {$_{a}$};
  \draw [->] (553bp,64bp) .. controls (558bp,68bp) and (562bp,72bp)  .. (566bp,77bp) .. controls (585bp,98bp) and (605bp,123bp)  .. (625bp,149bp);
  \draw (588bp,108bp) node {$_{z}$};
  \draw [->] (66bp,18bp) .. controls (77bp,18bp) and (90bp,18bp)  .. (112bp,18bp);
  \draw (89bp,24bp) node {$_{}$};
  \draw [->] (358bp,147bp) .. controls (368bp,157bp) and (383bp,172bp)  .. (396bp,185bp) .. controls (404bp,193bp) and (413bp,202bp)  .. (428bp,218bp);
  \draw (392bp,188bp) node {$_{j}$};
  \draw [->] (172bp,18bp) .. controls (183bp,18bp) and (196bp,18bp)  .. (218bp,18bp);
  \draw (195bp,24bp) node {$_{}$};
  \draw [->] (273bp,84bp) .. controls (284bp,83bp) and (296bp,82bp)  .. (318bp,81bp);
  \draw (295bp,89bp) node {$_{d}$};
  \draw [->] (368bp,26bp) .. controls (379bp,26bp) and (392bp,27bp)  .. (415bp,28bp);
  \draw (392bp,34bp) node {$_{n}$};
  \draw [->] (553bp,252bp) .. controls (570bp,235bp) and (597bp,206bp)  .. (624bp,179bp);
  \draw (588bp,226bp) node {$_{x}$};
  \draw [->] (459bp,109bp) .. controls (474bp,120bp) and (496bp,135bp)  .. (522bp,153bp);
  \draw (490bp,137bp) node {$_{t}$};
  \draw [->] (467bp,27bp) .. controls (476bp,28bp) and (486bp,28bp)  .. (494bp,30bp) .. controls (500bp,31bp) and (505bp,33bp)  .. (520bp,39bp);
  \draw (490bp,35bp) node {$_{w}$};
  \draw [->] (358bp,119bp) .. controls (362bp,115bp) and (366bp,110bp)  .. (370bp,106bp) .. controls (388bp,87bp) and (407bp,66bp)  .. (428bp,43bp);
  \draw (392bp,91bp) node {$_{l}$};
  \draw [->] (272bp,235bp) .. controls (283bp,235bp) and (295bp,235bp)  .. (317bp,235bp);
  \draw (295bp,242bp) node {$_{b}$};
  \draw [->] (464bp,226bp) .. controls (474bp,222bp) and (485bp,217bp)  .. (494bp,211bp) .. controls (504bp,204bp) and (513bp,195bp)  .. (527bp,179bp);
  \draw (490bp,219bp) node {$_{q}$};
  \draw [->] (465bp,287bp) .. controls (477bp,284bp) and (492bp,279bp)  .. (515bp,273bp);
  \draw (490bp,286bp) node {$_{o}$};
  \draw [->] (268bp,97bp) .. controls (282bp,104bp) and (299bp,112bp)  .. (323bp,123bp);
  \draw (295bp,115bp) node {$_{c}$};
  \draw [->] (361bp,223bp) .. controls (375bp,214bp) and (396bp,200bp)  .. (422bp,184bp);
  \draw (392bp,211bp) node {$_{i}$};
  \draw [->] (363bp,247bp) .. controls (377bp,255bp) and (397bp,268bp)  .. (422bp,283bp);
  \draw (392bp,272bp) node {$_{g}$};
  \draw [->] (458bp,220bp) .. controls (470bp,209bp) and (486bp,194bp)  .. (494bp,177bp) .. controls (515bp,136bp) and (492bp,117bp)  .. (512bp,77bp) .. controls (513bp,75bp) and (514bp,73bp)  .. (522bp,63bp);
  \draw (490bp,195bp) node {$_{r}$};
  \draw [->] (461bp,87bp) .. controls (474bp,80bp) and (492bp,72bp)  .. (517bp,60bp);
  \draw (490bp,79bp) node {$_{v}$};
  \draw [->] (365bp,125bp) .. controls (378bp,120bp) and (394bp,114bp)  .. (418bp,105bp);
  \draw (392bp,122bp) node {$_{k}$};
  \draw [->] (467bp,97bp) .. controls (476bp,97bp) and (485bp,97bp)  .. (494bp,98bp) .. controls (497bp,98bp) and (499bp,98bp)  .. (512bp,100bp);
  \draw (490bp,103bp) node {$_{u}$};
  \draw [->] (565bp,165bp) .. controls (576bp,165bp) and (589bp,165bp)  .. (611bp,165bp);
  \draw (588bp,172bp) node {$_{y}$};
  \draw [->] (468bp,168bp) .. controls (474bp,167bp) and (480bp,166bp)  .. (486bp,166bp) .. controls (491bp,165bp) and (497bp,165bp)  .. (513bp,165bp);
  \draw (490bp,171bp) node {$_{s}$};
  \draw [->] (464bp,241bp) .. controls (477bp,245bp) and (493bp,251bp)  .. (517bp,259bp);
  \draw (490bp,255bp) node {$_{p}$};
  \draw [->] (363bp,69bp) .. controls (373bp,65bp) and (385bp,62bp)  .. (396bp,66bp) .. controls (404bp,68bp) and (411bp,72bp)  .. (426bp,83bp);
  \draw (392bp,71bp) node {$_{m}$};

  \pgfsetstrokecolor{red}
  \draw (468bp,229bp) -- (463bp,244bp) -- (441bp,251bp) -- (419bp,244bp) -- (414bp,229bp) -- (429bp,217bp) -- (453bp,217bp) -- cycle;
  \pgfsetstrokecolor{black}
  \draw (441bp,233bp) node {$125$};


  \pgfsetstrokecolor{red}
  \draw (370bp,21bp) -- (365bp,36bp) -- (343bp,43bp) -- (321bp,36bp) -- (316bp,21bp) -- (331bp,9bp) -- (355bp,9bp) -- cycle;
  \pgfsetstrokecolor{black}
  \draw (343bp,25bp) node {$1115$};


  \pgfsetstrokecolor{red}
  \draw (566bp,46bp) -- (561bp,61bp) -- (539bp,68bp) -- (517bp,61bp) -- (512bp,46bp) -- (527bp,34bp) -- (551bp,34bp) -- cycle;
  \pgfsetstrokecolor{black}
  \draw (539bp,50bp) node {$17$};


  \pgfsetstrokecolor{red}
  \draw (566bp,100bp) -- (561bp,115bp) -- (539bp,122bp) -- (517bp,115bp) -- (512bp,100bp) -- (527bp,88bp) -- (551bp,88bp) -- cycle;
  \pgfsetstrokecolor{black}
  \draw (539bp,104bp) node {$44$};


  \pgfsetstrokecolor{red}
  \draw (468bp,25bp) -- (463bp,40bp) -- (441bp,47bp) -- (419bp,40bp) -- (414bp,25bp) -- (429bp,13bp) -- (453bp,13bp) -- cycle;
  \pgfsetstrokecolor{black}
  \draw (441bp,29bp) node {$116$};


  \pgfsetstrokecolor{red}
  \draw (370bp,231bp) -- (365bp,246bp) -- (343bp,253bp) -- (321bp,246bp) -- (316bp,231bp) -- (331bp,219bp) -- (355bp,219bp) -- cycle;
  \pgfsetstrokecolor{black}
  \draw (343bp,235bp) node {$1223$};


  \pgfsetstrokecolor{red}
  \draw (273bp,14bp) -- (267bp,29bp) -- (245bp,36bp) -- (223bp,29bp) -- (217bp,14bp) -- (233bp,2bp) -- (257bp,2bp) -- cycle;
  \pgfsetstrokecolor{black}
  \draw (245bp,18bp) node {$11114$};


  \pgfsetstrokecolor{red}
  \draw (566bp,161bp) -- (561bp,176bp) -- (539bp,183bp) -- (517bp,176bp) -- (512bp,161bp) -- (527bp,149bp) -- (551bp,149bp) -- cycle;
  \pgfsetstrokecolor{black}
  \draw (539bp,165bp) node {$35$};


  \pgfsetstrokecolor{red}
  \draw (468bp,168bp) -- (463bp,183bp) -- (441bp,190bp) -- (419bp,183bp) -- (414bp,168bp) -- (429bp,156bp) -- (453bp,156bp) -- cycle;
  \pgfsetstrokecolor{black}
  \draw (441bp,172bp) node {$233$};


  \pgfsetstrokecolor{red}
  \draw (566bp,263bp) -- (561bp,278bp) -- (539bp,285bp) -- (517bp,278bp) -- (512bp,263bp) -- (527bp,251bp) -- (551bp,251bp) -- cycle;
  \pgfsetstrokecolor{black}
  \draw (539bp,267bp) node {$26$};


  \pgfsetstrokecolor{red}
  \draw (664bp,161bp) -- (659bp,176bp) -- (637bp,183bp) -- (615bp,176bp) -- (610bp,161bp) -- (625bp,149bp) -- (649bp,149bp) -- cycle;
  \pgfsetstrokecolor{black}
  \draw (637bp,165bp) node {$8$};


  \pgfsetstrokecolor{red}
  \draw (273bp,231bp) -- (267bp,246bp) -- (245bp,253bp) -- (223bp,246bp) -- (217bp,231bp) -- (233bp,219bp) -- (257bp,219bp) -- cycle;
  \pgfsetstrokecolor{black}
  \draw (245bp,235bp) node {$11222$};


  \pgfsetstrokecolor{red}
  \draw (468bp,93bp) -- (463bp,108bp) -- (441bp,115bp) -- (419bp,108bp) -- (414bp,93bp) -- (429bp,81bp) -- (453bp,81bp) -- cycle;
  \pgfsetstrokecolor{black}
  \draw (441bp,97bp) node {$134$};


  \pgfsetstrokecolor{red}
  \draw (273bp,82bp) -- (267bp,97bp) -- (245bp,104bp) -- (223bp,97bp) -- (217bp,82bp) -- (233bp,70bp) -- (257bp,70bp) -- cycle;
  \pgfsetstrokecolor{black}
  \draw (245bp,86bp) node {$11123$};


  \pgfsetstrokecolor{red}
  \draw (370bp,75bp) -- (365bp,90bp) -- (343bp,97bp) -- (321bp,90bp) -- (316bp,75bp) -- (331bp,63bp) -- (355bp,63bp) -- cycle;
  \pgfsetstrokecolor{black}
  \draw (343bp,79bp) node {$1133$};


  \pgfsetstrokecolor{red}
  \draw (370bp,129bp) -- (365bp,144bp) -- (343bp,151bp) -- (321bp,144bp) -- (316bp,129bp) -- (331bp,117bp) -- (355bp,117bp) -- cycle;
  \pgfsetstrokecolor{black}
  \draw (343bp,133bp) node {$1124$};


  \pgfsetstrokecolor{red}
  \draw (173bp,14bp) -- (167bp,29bp) -- (142bp,36bp) -- (117bp,29bp) -- (111bp,14bp) -- (128bp,2bp) -- (156bp,2bp) -- cycle;
  \pgfsetstrokecolor{black}
  \draw (142bp,18bp) node {$111113$};


  \pgfsetstrokecolor{red}
  \draw (468bp,290bp) -- (463bp,305bp) -- (441bp,312bp) -- (419bp,305bp) -- (414bp,290bp) -- (429bp,278bp) -- (453bp,278bp) -- cycle;
  \pgfsetstrokecolor{black}
  \draw (441bp,294bp) node {$224$};


  \pgfsetstrokecolor{red}
  \draw (173bp,82bp) -- (167bp,97bp) -- (142bp,104bp) -- (117bp,97bp) -- (111bp,82bp) -- (128bp,70bp) -- (156bp,70bp) -- cycle;
  \pgfsetstrokecolor{black}
  \draw (142bp,86bp) node {$111122$};


  \pgfsetstrokecolor{red}
  \draw (68bp,14bp) -- (61bp,29bp) -- (34bp,36bp) -- (7bp,29bp) -- (0bp,14bp) -- (19bp,2bp) -- (49bp,2bp) -- cycle;
  \pgfsetstrokecolor{black}
  \draw (34bp,18bp) node {$1111112$};

\end{tikzpicture}
\end{figure}
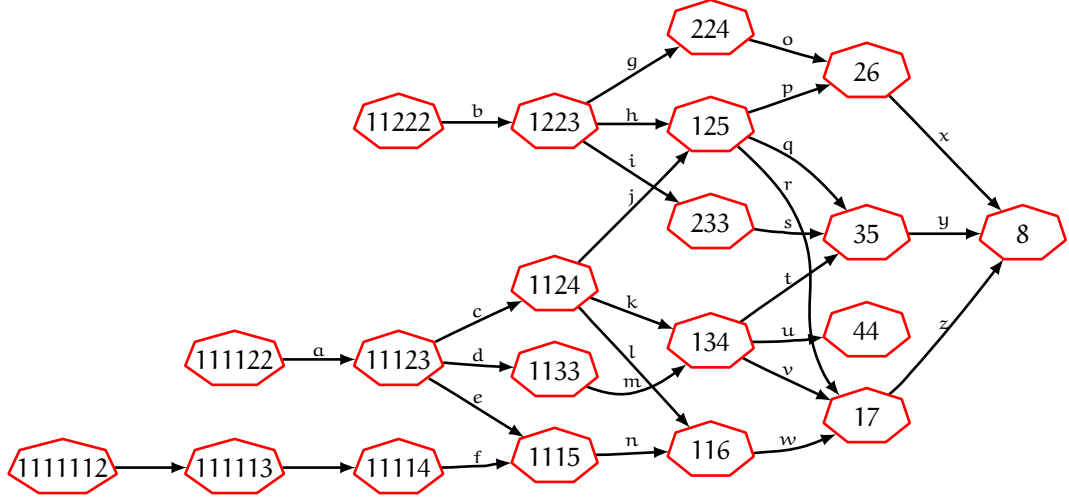

We define a map $\iota:Q_n\to k\mathcal{L}_n$ as follows.
Recall from \autoref{CompositionSection} that
we regard a partition as the P\'olya equivalence class of a composition,
but that we represent a partition by any convenient representative.
If $p$ is a vertex of $Q_n$ then we define
$\iota\left(p\right)$ to simply be $p$ itself,
now regarded as a class sum in $\mathcal{L}_n$. If $e$ is the edge
from a vertex $p$ to another vertex $q$, then by rearranging the
parts of $p$ we have
$p=p_{11}p_{12}p_2p_3\cdots p_j$ and $q=p_1p_2p_3\cdots p_j$
for some $j\in\mathbb{N}$ and some
$p_{11},p_{12},p_2,p_3,\ldots,p_j\in\mathbb{N}$
where $p_{11}<p_{12}$ and $p_1=p_{11}+p_{12}$. We put
$\iota\left(e\right)=
\left[\vcenter{\begin{xy}<.3cm,0cm>:
(2,2)="1"*+!U{_1};
"1";(1,1)*+!U{\color{red}p_{11}}**\dir{-};
"1";(3,1)*+!U{\color{red}p_{12}}**\dir{-};
(4,2)*+!UL{{\color{red}p_2p_3}\cdots{\color{red}p_j}};
\end{xy}}\right]$.
Note that $\iota$ satisfies
$\iota\left(xy\right)=\iota\left(y\right)\iota\left(x\right)$
whenever one of $x$ or $y$ is a vertex and the other
is an incident vertex or edge. This proves the following proposition.
\begin{proposition}\label{IotaHomomorphism}
The map $\iota$ extends to an anti-homomorphism $\iota:kQ_n\to k\mathcal{L}_n$.
\end{proposition}

We show in \autoref{IotaInjective} that $\iota$
is injective and
in \autoref{ExtQuiver}
that $Q_n$ is the ordinary quiver of $\Sigma\left(\sym{n}\right)$.
One of the main ingredients in the proof of \autoref{ExtQuiver}
is the following technical lemma.

\begin{lemma}\label{NodesInImage}
If $e$ is an edge  of $Q_n$
then $\iota\left(e\right)\not\in\ker\Delta$.
\end{lemma}

\begin{proof} Suppose
$\iota\left(e\right)=
\left[\vcenter{\begin{xy}<.2cm,0cm>:
(2,2)="1"*+!U{_1};
"1";(1,1)*+!U{\color{red}a}**\dir{-};
"1";(3,1)*+!U{\color{red}b}**\dir{-};
(4,2)*+!UL{{\color{red}q_1q_2}\cdots{\color{red}q_j}};
\end{xy}}\right]$ and that
$0\le i\le j$ is such that
$q_1\le q_2\le\cdots\le q_i\le a<q_{i+1}\le\cdots\le q_j$.
Then the term $q_1q_2\cdots q_iabq_{i+1}\cdots q_j$ of
$\Delta\left(\vcenter{\begin{xy}<.2cm,0cm>:
(1,2)*+!UR{{\color{red}q_1q_2}\cdots{\color{red}q_i}};
(3,2)="1"*+!U{_1};
"1";(2,1)*+!U{\color{red}a}**\dir{-};
"1";(4,1)*+!U{\color{red}b}**\dir{-};
(5,2)*+!UL{{\color{red}q_{i+1}}\cdots{\color{red}q_j}};
\end{xy}}\right)$
has at most one descending subsequence, namely $bq_{i+1}$.
However, all the terms of $\Delta\left(\iota\left(e\right)\right)$
appearing with negative coefficients have the descending subsequence
$ba$ which is different from $bq_{i+1}$ since $a<q_{i+1}$.
Thus $\Delta\left(\iota\left(e\right)\right)$ cannot be zero.
\end{proof}

In an effort both to simplify notation and to shift
emphasis from the individual groups $\sym{n}$ to the family
$\bigcup_{n\in\mathbb{N}\cup\left\{0\right\}}\sym{n}$ of groups, we define
\[Q=\coprod_{n\in\mathbb{N}\cup\left\{0\right\}} Q_n\qquad
L=\coprod_{n\in\mathbb{N}\cup\left\{0\right\}}L_n\qquad
\mathcal{L}=\coprod_{n\in\mathbb{N}\cup\left\{0\right\}}\mathcal{L}_n\]
and regard $\iota$ as a map $kQ\to k\mathcal{L}$.

\section{The branch monoid}\label{BranchSection}
Let $\mathcal{B}$ be the set of symbols
$\bb{a}{b}$ for all $a,b\in\mathbb{N}$ with $a<b$.
We call the free monoid $\mathcal{B}^\ast$ the {\em branch monoid}
and we write the element
$\bb{a_1}{b_1}\bb{a_2}{b_2}\cdots\bb{a_l}{b_l}$ of $\mathcal{B}^\ast$ as
$\lb{a_1}{b_1}\mb{a_2}{b_2}\cdots\rb{a_l}{b_l}$ to simplify notation.
The notation is meant to reflect the fact that
the elements of $\mathcal{B}^\ast$
can be used to build forests as we now describe.

If $X$ is a labeled forest then let
$X.\bb{a}{b}$ be the sum of all forests
that can be obtained from $X$ by replacing a leaf $a+b$ with
$\vcenter{\begin{xy}<.2cm,0cm>:
(2,3)="1"*+!U{_l};
"1";(1,2)*+!U{\color{red}a}**\dir{-};
"1";(3,2)*+!U{\color{red}b}**\dir{-};
\end{xy}}$
where $l=\ell\left(X\right)+1$.
If $P$ is a path in $Q$ with source $p$ then we define $P.\bb{a}{b}$
to be the path obtained from $P$ by appending the edge 
from $abq_1\cdots q_j$ to $p$
if $p$ has a part $a+b$, where
$q_1,\ldots,q_j\in\mathbb{N}$ are the remaining parts of $p$.
We put $P.\bb{a}{b}=0$ if $p$ has no part $a+b$.
Then $\mathcal{B}^\ast$ acts on $kL$ and on $kQ$ by extending
the definitions above by linearity.
From the definitions we have
\begin{equation}\label{DProduct}
\left(P_1P_2\right).B=\left(P_1.B\right)P_2
\qquad\text{and}\qquad \left(X_1\bullet X_2\right).B=X_1\bullet X_2.B
\end{equation}
for $P_1,P_2\in kQ$ and $X_1,X_2\in kL$ and $B\in\mathcal{B}^\ast$.
If $p$ is a partition containing a part $a+b$ then
\begin{equation}\label{PProduct}
\iota\left(p\right).\bb{a}{b}
=\left[\vcenter{\begin{xy}<.3cm,0cm>:
(2,2)="1"*+!U{_1};
"1";(1,1)*+!U{\color{red}a}**\dir{-};
"1";(3,1)*+!U{\color{red}b}**\dir{-};
(4,2)*+!UL{{\color{red}q_1q_2}\cdots{\color{red}q_j}};
\end{xy}}\right]=\iota\left(p.\bb{a}{b}\right)
\end{equation}
where $q_1,q_2,\ldots,q_j$ are the remaining parts of $p$.
On the other hand,
both $\iota\left(p\right).\bb{a}{b}$ and $p.\bb{a}{b}$ are zero
if $p$ has no part $a+b$.
Now if $P$ is a path in $Q$ with source $p$,
then using \autoref{DProduct} and \autoref{PProduct} we have
\[\iota\left(P\right).B
=\iota\left(pP\right).B
=\left(\iota\left(P\right)\bullet\iota\left(p\right)\right).B
=\iota\left(P\right)\bullet\iota\left(p.B\right)
=\iota\left(\left(p.B\right)P\right)
=\iota\left(P.B\right)\]
for all $B\in k\mathcal{B}^\ast$.
This proves the following proposition.
\begin{proposition}\label{IotaDHomomorphism}
The map $\iota$ is a homomorphism of $k\mathcal{B}^\ast$-modules.
\end{proposition}

The branch monoid provides a convenient language for specifying
paths in $Q$. Namely, we can uniquely specify any path $P$
as $p.B$ where $p$ is the destination of $P$ and $B$ is an element
of $\mathcal{B}^\ast$. Furthermore, the element $B$ is related to
$\iota\left(P\right)$ in the way described in the following lemma. 

\begin{lemma}\label{IotaEquivariant} Let $P=p.\lb{a_1}{b_1}\mb{a_2}{b_2}
\cdots\rb{a_l}{b_l}$ be a path in $Q$. Then the node 
$\vcenter{\begin{xy}<.3cm,0cm>:
(2,3)="1"*+!U{_i};
"1";(1,2)*+!U{Z_1}**\dir{-};
"1";(3,2)*+!U{Z_2}**\dir{-};
\end{xy}}$
of every term of $\iota\left(P\right)$
satisfies $\overline{Z_1}=a_i$
and $\overline{Z_2}=b_i$ for all $1\le i\le l$.
\end{lemma}

\begin{proof} The assertion holds by definition if $l$
equals zero or one. Otherwise let $P'=p.\lb{a_1}{b_1}\mb{a_2}{b_2}
\cdots\rb{a_{l-1}}{b_{l-1}}$ so that
$P=P'.\bb{a_l}{b_l}$
and $\iota\left(P\right)=\iota\left(P'\right).\bb{a_l}{b_l}$
by \autoref{IotaDHomomorphism}.
Then $\iota\left(P\right)$ is obtained from $\iota\left(P'\right)$ by replacing
a leaf $a_l+b_l$ in every term with
$\vcenter{\begin{xy}<.3cm,0cm>:
(2,2)="1"*+!U{_l};
"1";(1,1)*+!U{\color{red}a_l}**\dir{-};
"1";(3,1)*+!U{\color{red}b_l}**\dir{-};
\end{xy}}$. Thus the node labeled $l$ of every term of $\iota\left(P\right)$
satisfies the assertion,
while the other nodes satisfy the assertion by induction.
\end{proof}

\begin{corollary}\label{IotaInjective}
The anti-homomorphism $\iota$ is injective.
\end{corollary}

\begin{proof} By \autoref{IotaEquivariant} the images of distinct
paths are supported on disjoint subsets of $\mathcal{L}$.
\end{proof}

\section{Unlabeled forests}\label{UnlabeledSection}
To compute the kernel of $\Delta:k\mathcal{L}\to k\mathbb{N}^\ast$
it will be helpful
to introduce an algebra through which $\Delta$ factors.
Then the kernel of $\Delta$ can be assembled from
the kernels of its factors.
Let $M$ be the category of
{\em unlabeled forests}, which are simply sequences
of binary trees whose leaves are natural numbers.
The definitions of the foliage, squash, length,
value, and product of unlabeled forests can be easily adapted
from the definitions for labeled forests, as can the P\'olya action
and the action of $k\mathcal{B}^\ast$ on $M$. Then 
\[M=\coprod_{n\in\mathbb{N}\cup\left\{0\right\}} M_n
\qquad\text{and}\qquad
\mathcal{M}=\coprod_{n\in\mathbb{N}\cup\left\{0\right\}}\mathcal{M}_n\]
where $M_n$ is the category of unlabeled forests of value $n$
and $\mathcal{M}$ and $\mathcal{M}_n$ are the sets 
of P\'olya class sums in $kM$ and $kM_n$.

There is a map $\E:L\to M$ given by erasing
the node labels of a forest.
If $X$ is a labeled forest with $j$ parts, then we denote
by $\alpha_X$ the index
of the stabilizer of $X$ in $\sym{j}$ in the stabilizer
of $\E\left(X\right)$ in $\sym{j}$.

\begin{lemma}\label{IndexStabilizer} If $X\in L$ then $\E\left[X\right]
=\alpha_X\left[\E\left(X\right)\right]$.
\end{lemma}

For example, if $X=
\vcenter{\begin{xy}<.2cm,0cm>:
(2,2)="1"*+!U{_1};
"1";(1,1)*+!U{\color{red}1}**\dir{-};
"1";(3,1)*+!U{\color{red}2}**\dir{-};
(5,2)="2"*+!U{_2};
"2";(4,1)*+!U{\color{red}1}**\dir{-};
"2";(6,1)*+!U{\color{red}2}**\dir{-};
\end{xy}}$
then
$\left[X\right]=
\vcenter{\begin{xy}<.2cm,0cm>:
(2,2)="1"*+!U{_1};
"1";(1,1)*+!U{\color{red}1}**\dir{-};
"1";(3,1)*+!U{\color{red}2}**\dir{-};
(5,2)="2"*+!U{_2};
"2";(4,1)*+!U{\color{red}1}**\dir{-};
"2";(6,1)*+!U{\color{red}2}**\dir{-};
\end{xy}}
+\vcenter{\begin{xy}<.2cm,0cm>:
(2,2)="1"*+!U{_2};
"1";(1,1)*+!U{\color{red}1}**\dir{-};
"1";(3,1)*+!U{\color{red}2}**\dir{-};
(5,2)="2"*+!U{_1};
"2";(4,1)*+!U{\color{red}1}**\dir{-};
"2";(6,1)*+!U{\color{red}2}**\dir{-};
\end{xy}}$
while $\left[\E\left(X\right)\right]=
\vcenter{\begin{xy}<.2cm,0cm>:
(2,2)="1";
"1";(1,1)*+!U{\color{red}1}**\dir{-};
"1";(3,1)*+!U{\color{red}2}**\dir{-};
(5,2)="2";
"2";(4,1)*+!U{\color{red}1}**\dir{-};
"2";(6,1)*+!U{\color{red}2}**\dir{-};
\end{xy}}$
so that $\E\left[X\right]=2\left[\E\left(X\right)\right]$.

Recall that the product in $L$ of two forests is formed
by replacing the leaves in one forest with the trees of the other.
Since this process depends on the foliage and squash
but not the node labels of the two forests,
we observe that up to node label erasure, the same
products are formed with or without the node labels.
This means that $\E$ is a functor and that
the induced map $\E:kL\to kM$ is an algebra homomorphism.
Then since the restriction of $\E$
to the subalgebra $k\mathcal{L}$ has image
in $k\mathcal{M}$ by \autoref{IndexStabilizer} we have
the following result.
\begin{proposition}
The map $\E:k\mathcal{L}\to k\mathcal{M}$ given by erasing
node labels is an algebra homomorphism.
\end{proposition}

As with labeled forests, the definition of unlabeled forests
can be made mathematically precise
by defining the set of unlabeled trees $\mathbb{M}$
to be the minimal set containing
$\mathbb{N}$ and also containing the tuple
$\left(X_1,X_2\right)$ whenever
$X_1,X_2\in\mathbb{M}$.
Then for example, the map $\E$ can be defined by
\[\E\left(X\right)=\begin{cases}
X&\text{if $X\in\mathbb{N}$}\\
\left(\E\left(X_1\right),\E\left(X_2\right)\right)
&\text{if $X=\left(X_1,i,X_2\right)\in\mathbb{L}$}\end{cases}\]
and similarly for the other functions of unlabeled forests.

\section{Alignment}\label{AlignmentSection}
Let $\mathbb{M}$ be the free magma generated by $\mathbb{N}$.
We denote the product of two elements $X$ and $Y$ of $\mathbb{M}$ by
$\vcenter{\begin{xy}<.2cm,0cm>:
(1,2)="1";
"1";(0,1)*+!U{X}**\dir{-};
"1";(2,1)*+!U{Y}**\dir{-};
\end{xy}}$. Although we could introduce a symbol for this operation,
after several iterations, it becomes more legible 
to simply represent elements of $\mathbb{M}$
as binary trees. We define the ideals
\begin{align*}
N&=\left\langle
\vcenter{\begin{xy}<.2cm,0cm>:
(1,2)="1";
"1";(0,1)*+!U{X}**\dir{-};
"1";(2,1)*+!U{Y}**\dir{-};
\end{xy}}
+\vcenter{\begin{xy}<.2cm,0cm>:
(1,2)="1";
"1";(0,1)*+!U{Y}**\dir{-};
"1";(2,1)*+!U{X}**\dir{-};
\end{xy}}
\mid X,Y\in\mathbb{M}\right\rangle\\
J&=\left\langle
\vcenter{\begin{xy}<.2cm,0cm>:
(1,3)="1";
"1";(0,2)*+!U{X}**\dir{-};
"1";(3,2)="2"**\dir{-};
"2";(2,1)*+!U{Y}**\dir{-};
"2";(4,1)*+!U{Z}**\dir{-};
\end{xy}}
+\vcenter{\begin{xy}<.2cm,0cm>:
(1,3)="1";
"1";(0,2)*+!U{Y}**\dir{-};
"1";(3,2)="2"**\dir{-};
"2";(2,1)*+!U{Z}**\dir{-};
"2";(4,1)*+!U{X}**\dir{-};
\end{xy}}
+\vcenter{\begin{xy}<.2cm,0cm>:
(1,3)="1";
"1";(0,2)*+!U{Z}**\dir{-};
"1";(3,2)="2"**\dir{-};
"2";(2,1)*+!U{X}**\dir{-};
"2";(4,1)*+!U{Y}**\dir{-};
\end{xy}}
\mid X,Y,Z\in\mathbb{M}\right\rangle
\end{align*}
of $k\mathbb{M}$ and recall that
$k\mathbb{M}/\left(N+J\right)$ defines the free Lie algebra
over $k$ generated by~$\mathbb{N}$.

Since the elements of $\mathbb{M}$ correspond with elements
of $M$ that have exactly one part,
we can identify arbitrary elements of $M$ with the elements
of the free monoid $\mathbb{M}^\ast$.
Under this identification, the category $M$
has, in addition to the product $\bullet$,
another product coming from concatenation
in $\mathbb{M}^\ast$.
Let $\mathcal{N}$ and $\mathcal{J}$
be the ideals of $kM$ with respect to concatenation
generated by $N$ and $J$ respectively.

Let $\pi:k\mathbb{M}\to k\mathbb{N}^\ast$ be defined by
$\pi\left(x\right)=x$ for $x\in\mathbb{N}$ and $\pi\left(
\vcenter{\begin{xy}<.2cm,0cm>:
(1,2)="1";
"1";(0,1)*+!U{X}**\dir{-};
"1";(2,1)*+!U{Y}**\dir{-};
\end{xy}}\right)
=\pi\left(X\right)\pi\left(Y\right)-\pi\left(Y\right)\pi\left(X\right)$
for $X,Y\in\mathbb{M}$.
Then $\pi$ extends to
a monoid algebra homomorphism $\pi:kM\to k\mathbb{N}^\ast$ and
the kernel of $\pi$ is the ideal $\mathcal{N}+\mathcal{J}$
generated by the kernel $N+J$ of $\pi:k\mathbb{M}\to k\mathbb{N}^\ast$.
Recall that the map $\Delta$ replaces the nodes of a labeled tree with Lie
brackets in the order specified by the node labels. The relationship between
$\Delta$ and $\pi$ is given in the following lemma.

\begin{lemma}\label{DeltaPiE} $\Delta=\pi\circ\E$\end{lemma}
\begin{proof} Let $X\in L$. Then
$\Delta\left(X\right)=\pi\left(\E\left(X\right)\right)$
by definition if $X$ has length zero.
Suppose otherwise that $X=X_1X_2\cdots X_j$ where $X_1,X_2,\ldots,X_j$
are trees and suppose that the node of $X$ labeled $1$ is $X_i$ so that
$X_i=\vcenter{\begin{xy}<.3cm,0cm>:
(1,2)="1"*+!U{_1};
"1";(0,1)*+!U{X_{i1}}**\dir{-};
"1";(2,1)*+!U{X_{i2}}**\dir{-};
\end{xy}}$ for some trees $X_{i1}$ and $X_{i2}$. Then
\begin{align*}
\pi\left(\E\left(X\right)\right)
&=\pi\left(\E\left(X_1\right)\right)
\cdots\pi\left(\E\left(X_i\right)\right)
\cdots\pi\left(\E\left(X_j\right)\right)\\
&=\pi\left(\E\left(X_1\right)\right)
\cdots\pi\left(\E\left(X_{i1}X_{i2}-X_{i2}X_{i1}\right)\right)
\cdots\pi\left(\E\left(X_j\right)\right)\\
&=\pi\left(\E\left(X_1\cdots\left(X_{i1}X_{i2}-X_{i2}X_{i1}\right)\cdots X_j\right)\right)\\
&=\pi\left(\E\left(\delta\left(X\right)\right)\right).
\end{align*}
Then $\pi\left(\E\left(\delta\left(X\right)\right)\right)
=\Delta\left(\delta\left(X\right)\right)=\Delta\left(X\right)$
by induction since $\delta\left(X\right)$ has shorter length than $X$.
\end{proof}

A forest $X$ is called {\em aligned} if $\overline{Z_1}<\overline{Z_2}$
for all nodes 
$\vcenter{\begin{xy}<.3cm,0cm>:
(1,2)="1";
"1";(0,1)*+!U{Z_1}**\dir{-};
"1";(2,1)*+!U{Z_2}**\dir{-};
\end{xy}}$
of $X$. Since the product of two
aligned forests is aligned,
the category $M^+$ of aligned unlabeled
forests is a subcategory of $M$ and
\[M^+=\coprod_{n\in\mathbb{N}\cup\left\{0\right\}} M^+_n
\qquad\text{and}\qquad
\mathcal{M}^+=\coprod_{n\in\mathbb{N}\cup\left\{0\right\}}\mathcal{M}^+_n\]
where $M^+_n$ is the category of aligned unlabeled forests of value $n$
and $\mathcal{M}^+$ and $\mathcal{M}^+_n$ are the sets of class sums
in $kM^+$ and $kM^+_n$.
We similarly define the corresponding sets of aligned labeled forests
$L^+$, $L_n^+$, $\mathcal{L}^+$, $\mathcal{L}_n^+$.
Our first observation about aligned forests is
that the image of $\iota$ is aligned.

\begin{lemma}\label{ImageIotaAligned}
$\iota\left(kQ\right)\subseteq k\mathcal{L}^+$
\end{lemma}

\begin{proof}
We observe that $\iota\left(e\right)$
is aligned for each edge $e$ of $Q$ 
as a result of the requirement $p_{11}<p_{12}$ in
the definition of $\iota$.
Then since $\iota$ is an anti-homomorphism by \autoref{IotaHomomorphism}
it follows that the image of every element of $kQ$ under $\iota$ is aligned.
\end{proof}

\begin{lemma}\label{AlignedRendering}
If $X\in M$ and no node 
$\vcenter{\begin{xy}<.25cm,0cm>:
(1,2)="1";
"1";(0,1)*+!U{Z_1}**\dir{-};
"1";(2,1)*+!U{Z_2}**\dir{-};
\end{xy}}$
of $X$ satisfies $Z_1=Z_2\in\mathbb{N}$
then there exist $A\in kM^+$ and $Y\in\mathcal{N}+\mathcal{J}$
such that $A=X+Y$.
\end{lemma}

\begin{proof}
If $X$ is aligned, then we can take $A=X$ and $Y=0$.
Otherwise let $Z=\vcenter{\begin{xy}<.25cm,0cm>:
(1,2)="1";
"1";(0,1)*+!U{Z_1}**\dir{-};
"1";(2,1)*+!U{Z_2}**\dir{-};
\end{xy}}$ be a node of $X$ for which
$\overline{Z_1}\ge\overline{Z_2}$.
We define an auxiliary element $X'\in kM$ as follows.
If $\overline{Z_1}>\overline{Z_2}$
then let $X'$ be the forest obtained from $X$ by exchanging
$Z_1$ with $Z_2$ so that $X+X'\in\mathcal{N}$.
Otherwise suppose that
$\overline{Z_1}=\overline{Z_2}$.
Observe that one of $Z_1$ or $Z_2$ has positive length by hypothesis.
If $\ell\left(Z_2\right)>0$ then
$Z=\vcenter{\begin{xy}<.3cm,0cm>:
(1,3)="1";
"1";(0,2)*+!U{Z_1}**\dir{-};
"1";(3,2)="2"**\dir{-};
"2";(2,1)*+!U{Z_{21}}**\dir{-};
"2";(4,1)*+!U{Z_{22}}**\dir{-};
\end{xy}}$ 
for some trees $Z_{21}$ and $Z_{22}$.
Let $X'$ be obtained from $X$ by replacing $Z$ with
$\vcenter{\begin{xy}<.3cm,0cm>:
(1,3)="1";
"1";(0,2)*+!U{Z_{22}}**\dir{-};
"1";(3,2)="2"**\dir{-};
"2";(2,1)*+!U{Z_1}**\dir{-};
"2";(4,1)*+!U{Z_{21}}**\dir{-};
\end{xy}}
+\vcenter{\begin{xy}<.3cm,0cm>:
(1,3)="1";
"1";(0,2)*+!U{Z_{21}}**\dir{-};
"1";(3,2)="2"**\dir{-};
"2";(2,1)*+!U{Z_{22}}**\dir{-};
"2";(4,1)*+!U{Z_1}**\dir{-};
\end{xy}}$
so that $X+X'\in\mathcal{J}$.
If $\ell\left(Z_1\right)>0$
then we can apply both replacements above to define
an element $X'$ such that $X+X'\in\mathcal{N}+\mathcal{J}$.

Observe that each term of $X'$ has fewer nodes
$\vcenter{\begin{xy}<.25cm,0cm>:
(1,2)="1";
"1";(0,1)*+!U{U_1}**\dir{-};
"1";(2,1)*+!U{U_2}**\dir{-};
\end{xy}}$
with $\overline{U_1}\ge\overline{U_2}$ than $X$. Then by induction
$A'=X'+Y'$ for some $A'\in kM^+$ and some $Y'\in\mathcal{N}+\mathcal{J}$.
Taking $A=-A'$ and $Y=-X-X'-Y'$ gives the result.
\end{proof}

The element $A$ in \autoref{AlignedRendering}
is called an {\em aligned rendering} of $X$.
An aligned rendering of a forest
need not be unique.
For example, the forest
$\vcenter{\begin{xy}<.2cm,0cm>:
(2,4)="1";
"1";(1,3)*+!U{\color{red}6}**\dir{-};
"1";(6,3)="2"**\dir{-};
"2";(4,2)="3"**\dir{-};
"3";(3,1)*+!U{\color{red}1}**\dir{-};
"3";(5,1)*+!U{\color{red}2}**\dir{-};
"2";(7,2)*+!U{\color{red}3}**\dir{-};
\end{xy}}$
has aligned renderings
\begin{equation}\label{TwoAlignedRenderings}
\vcenter{\begin{xy}<.2cm,0cm>:
(2,4)="1";
"1";(1,3)*+!U{\color{red}3}**\dir{-};
"1";(6,3)="2"**\dir{-};
"2";(4,2)="3"**\dir{-};
"3";(3,1)*+!U{\color{red}1}**\dir{-};
"3";(5,1)*+!U{\color{red}2}**\dir{-};
"2";(7,2)*+!U{\color{red}6}**\dir{-};
\end{xy}}
-\vcenter{\begin{xy}<.2cm,0cm>:
(4,3)="1";
"1";(2,2)="2"**\dir{-};
"2";(1,1)*+!U{\color{red}1}**\dir{-};
"2";(3,1)*+!U{\color{red}2}**\dir{-};
"1";(6,2)="3"**\dir{-};
"3";(5,1)*+!U{\color{red}3}**\dir{-};
"3";(7,1)*+!U{\color{red}6}**\dir{-};
\end{xy}}
\quad\text{and}\quad
\vcenter{\begin{xy}<.2cm,0cm>:
(2,4)="1";
"1";(1,3)*+!U{\color{red}2}**\dir{-};
"1";(6,3)="2"**\dir{-};
"2";(4,2)="3"**\dir{-};
"3";(3,1)*+!U{\color{red}1}**\dir{-};
"3";(5,1)*+!U{\color{red}3}**\dir{-};
"2";(7,2)*+!U{\color{red}6}**\dir{-};
\end{xy}}
-\vcenter{\begin{xy}<.2cm,0cm>:
(4,3)="1";
"1";(2,2)="2"**\dir{-};
"2";(1,1)*+!U{\color{red}1}**\dir{-};
"2";(3,1)*+!U{\color{red}3}**\dir{-};
"1";(6,2)="3"**\dir{-};
"3";(5,1)*+!U{\color{red}2}**\dir{-};
"3";(7,1)*+!U{\color{red}6}**\dir{-};
\end{xy}}
-\vcenter{\begin{xy}<.2cm,0cm>:
(2,4)="1";
"1";(1,3)*+!U{\color{red}1}**\dir{-};
"1";(6,3)="2"**\dir{-};
"2";(4,2)="3"**\dir{-};
"3";(3,1)*+!U{\color{red}2}**\dir{-};
"3";(5,1)*+!U{\color{red}3}**\dir{-};
"2";(7,2)*+!U{\color{red}6}**\dir{-};
\end{xy}}
+\vcenter{\begin{xy}<.2cm,0cm>:
(4,3)="1";
"1";(2,2)="2"**\dir{-};
"2";(1,1)*+!U{\color{red}2}**\dir{-};
"2";(3,1)*+!U{\color{red}3}**\dir{-};
"1";(6,2)="3"**\dir{-};
"3";(5,1)*+!U{\color{red}1}**\dir{-};
"3";(7,1)*+!U{\color{red}6}**\dir{-};
\end{xy}}
\end{equation}
obtained by applying the replacements in
\autoref{AlignedRendering} to different nodes.

\section{The path associated to a forest}
Continuing the example at the beginning of \autoref{QuiverSection}
we recall that $Q$ was constructed on the basis of the unique
factorization of labeled forests.
However, when mapping the path algebra
of $Q$ back to the algebra of labeled forests,
we replaced the factors in such a factorization with their
P\'olya classes, which are more useful
in light of our interest in
$\Sigma\left(\sym{n}\right)$ but which break the 
factorization, as the example shows.
Specifically we associated
the path $P=p.\lb{1}{3}\rb{1}{2}$ to the class
$\left[\vcenter{\begin{xy}<.2cm,0cm>:
(2,3)="1"*+!U{_1};
"1";(1,2)*+!U{\color{red}1}**\dir{-};
"1";(4,2)="2"*+!U{_2}**\dir{-};
"2";(3,1)*+!U{\color{red}1}**\dir{-};
"2";(5,1)*+!U{\color{red}2}**\dir{-};
(6,3)*+!U{\color{red}3};
\end{xy}}\right]$
where $p$ is the partition $34$
and we found that $\iota\left(P\right)
=\left[\vcenter{\begin{xy}<.2cm,0cm>:
(2,3)="1"*+!U{_1};
"1";(1,2)*+!U{\color{red}1}**\dir{-};
"1";(4,2)="2"*+!U{_2}**\dir{-};
"2";(3,1)*+!U{\color{red}1}**\dir{-};
"2";(5,1)*+!U{\color{red}2}**\dir{-};
(6,3)*+!U{\color{red}3};
\end{xy}}\right]
+\left[\vcenter{\begin{xy}<.2cm,0cm>:
(2,2)="1"*+!U{_1};
"1";(1,1)*+!U{\color{red}1}**\dir{-};
"1";(3,1)*+!U{\color{red}3}**\dir{-};
(5,2)="1"*+!U{_2};
"1";(4,1)*+!U{\color{red}1}**\dir{-};
"1";(6,1)*+!U{\color{red}2}**\dir{-};
\end{xy}}\right]$.
Applying the same procedure instead to
$\left[\vcenter{\begin{xy}<.2cm,0cm>:
(2,2)="1"*+!U{_1};
"1";(1,1)*+!U{\color{red}1}**\dir{-};
"1";(3,1)*+!U{\color{red}3}**\dir{-};
(5,2)="1"*+!U{_2};
"1";(4,1)*+!U{\color{red}1}**\dir{-};
"1";(6,1)*+!U{\color{red}2}**\dir{-};
\end{xy}}\right]$ results in the same path $P$,
so again the factorization fails.
Motivated by this example, the purpose of this section is to precisely define
the path associated to a labeled forest and to calculate
its image under $\iota$. In \autoref{FactorizationSection}
we show how the failure of factorization in $\mathcal{L}$
can be resolved.

Consider the following transformations of a labeled forest.
\begin{enumerate}
\item\label{MoveOne} exchanging two subtrees $U$ and $V$
for which $\overline{U}=\overline{V}$
and the node labels
of the parents of $U$ and $V$, if they exist, are {\em smaller}
than the node labels of $U$ and $V$, if they exist
\item\label{MoveTwo} exchanging two parts of the forest
\end{enumerate}
Note that both moves produce another labeled forest.
We write $X\sim Y$
for $X,Y\in L$ if $Y$ can be obtained from $X$
by applying a sequence of moves \autoref{MoveOne} or \autoref{MoveTwo}.
Then $\sim$ is an equivalence relation on $L$
that induces an equivalence relation on $\mathcal{L}$.
Note that if $X\sim Y$
then $X$ is aligned if and only
if $Y$ is aligned. For example, the forests
\begin{equation}\label{SimExample}
\vcenter{\begin{xy}<.2cm,0cm>:
(2,4)="1"*+!U{_1};
"1";(1,3)*+!U{\color{red}4}**\dir{-};
"1";(6,3)="2"*+!U{_2}**\dir{-};
"2";(4,2)="4"*+!U{_4}**\dir{-};
"4";(3,1)*+!U{\color{red}1}**\dir{-};
"4";(5,1)*+!U{\color{red}2}**\dir{-};
"2";(8,2)="3"*+!U{_3}**\dir{-};
"3";(7,1)*+!U{\color{red}1}**\dir{-};
"3";(9,1)*+!U{\color{red}3}**\dir{-};
\end{xy}}\qquad
\vcenter{\begin{xy}<.2cm,0cm>:
(2,5)="1"*+!U{_1};
"1";(1,4)*+!U{\color{red}4}**\dir{-};
"1";(4,4)="2"*+!U{_2}**\dir{-};
"2";(3,3)*+!U{\color{red}3}**\dir{-};
"2";(6,3)="3"*+!U{_3}**\dir{-};
"3";(5,2)*+!U{\color{red}1}**\dir{-};
"3";(8,2)="4"*+!U{_4}**\dir{-};
"4";(7,1)*+!U{\color{red}1}**\dir{-};
"4";(9,1)*+!U{\color{red}2}**\dir{-};
\end{xy}}\qquad
\vcenter{\begin{xy}<.2cm,0cm>:
(4,4)="1"*+!U{_1};
"1";(2,3)="3"*+!U{_3}**\dir{-};
"3";(1,2)*+!U{\color{red}1}**\dir{-};
"3";(3,2)*+!U{\color{red}3}**\dir{-};
"1";(8,3)="2"*+!U{_2}**\dir{-};
"2";(6,2)="4"*+!U{_4}**\dir{-};
"4";(5,1)*+!U{\color{red}1}**\dir{-};
"4";(7,1)*+!U{\color{red}2}**\dir{-};
"2";(9,2)*+!U{\color{red}4}**\dir{-};
\end{xy}}\qquad
\vcenter{\begin{xy}<.2cm,0cm>:
(6,4)="1"*+!U{_1};
"1";(2,3)="3"*+!U{_3}**\dir{-};
"3";(1,2)*+!U{\color{red}1}**\dir{-};
"3";(4,2)="4"*+!U{_4}**\dir{-};
"4";(3,1)*+!U{\color{red}1}**\dir{-};
"4";(5,1)*+!U{\color{red}2}**\dir{-};
"1";(8,3)="2"*+!U{_2}**\dir{-};
"2";(7,2)*+!U{\color{red}3}**\dir{-};
"2";(9,2)*+!U{\color{red}4}**\dir{-};
\end{xy}}
\end{equation}
are related by $\sim$.

As in the example at the beginning of this section,
we can associate a path in $Q$ to an aligned labeled
forest through the map
$\p:L^+\to kQ$ defined as follows.
Suppose that $X$ is an aligned labeled forest
and let $p$ be the composition $\overline{X}$ regarded
as a vertex of $Q$. Let
$a_1,b_1,a_2,b_2,\ldots,a_l,b_l\in\mathbb{N}$
be such that the node
$\vcenter{\begin{xy}<.3cm,0cm>:
(2,3)="1"*+!U{_i};
"1";(1,2)*+!U{Z_1}**\dir{-};
"1";(3,2)*+!U{Z_2}**\dir{-};
\end{xy}}$
of $X$ satisfies $\s{Z_1}=a_i$
and $\s{Z_2}=b_i$ for all $1\le i\le l$
where $l=\ell\left(X\right)$.
Then we define $\p\left(X\right)
=p.\lb{a_1}{b_1}\mb{a_2}{b_2}\cdots\rb{a_l}{b_l}$.

We observe that applying $\p$ to forests
related to one another by $\sim$ produces the same path.
In particular, applying $\p$ to forests
in the same P\'olya class produces
the same path. Therefore we can define
$\p\left[X\right]=\p\left(X\right)$ for all $X\in L^+$.
Finally, applying $\p$ to any terms of
the image under $\iota$ of any path $P$
produces the same path
by \autoref{IotaEquivariant},
which must therefore be $P$.
For example, if $X$ is any of the forests in \autoref{SimExample}
then $\p\left[X\right]=p.\lb{4}{7}\mb{3}{4}\mb{1}{3}\rb{1}{2}$
where $p$ is the partition containing the single part eleven.

The map $\p$ can also be formulated recursively as follows.
Suppose again that $X$ is an aligned labeled forest.
If $X$ has length zero, then we can regard $X$ as a vertex of $Q$
and take $\p\left(X\right)=X$.
Otherwise we define
$\p\left(X\right)=\p\left(Y\right)e$
where $X',Y$ are as in \autoref{UniqueFactorization}
and $e$ is the edge of $Q$ from
$\f{X'}$ to $\s{X'}$.
Note that $\left[X'\right]=\iota\left(e\right)$
so that $\iota$ and $\p$
are inverses of one another when restricted to elements of length one.
The same is true of elements of length zero.
The following lemma deals with the composition $\iota\circ\p$ in general.

\begin{lemma}\label{IotaWiggle} If $X\in L^+$ then
$\displaystyle\iota\left(\p\left[X\right]\right)
=\sum_{\left[U\right]\sim\left[X\right]}\left[U\right]$.\end{lemma}

\begin{proof} As mentioned above $\iota\left(\p\left[X\right]\right)
=\left[X\right]$ if $X$ has length zero or one.
Otherwise let $X',Y$ be as in \autoref{UniqueFactorization}.
Assuming by induction that
$\iota\left(\p\left[Y\right]\right)
=\sum_{\left[V\right]\sim\left[Y\right]}\left[V\right]$ we have
\begin{equation}\label{IotaFactorization}
\iota\left(\p\left[X\right]\right)
=\iota\left(\p\left[X'\bullet Y\right]\right)
=\left[\vcenter{\begin{xy}<.3cm,0cm>:
(0,2)*+!UR{{\color{red}x_1x_2}\cdots{\color{red}x_{i-1}}};
(2,2)="1"*+!U{_1};
"1";(1,1)*+!U{{\color{red}x_{i1}}}**\dir{-};
"1";(3,1)*+!U{{\color{red}x_{i2}}}**\dir{-};
(4,2)*+!UL{{\color{red}x_{i+1}}\cdots{\color{red}x_j}};
\end{xy}}\right]
\bullet\sum_{\left[V\right]\sim\left[Y\right]}\left[V\right].
\end{equation}
Note that all the terms $\left[U\right]$ of \autoref{IotaFactorization}
satisfy $\left[U\right]\sim\left[X\right]$.
Conversely, suppose that $\left[U\right]$ is such that
$\left[U\right]\sim\left[X\right]$.
We can assume that $U$ can be obtained from $X$ by exchanging
a single pair of subtrees of the same squash since $\sim$
is the reflexive and transitive closure of the set of all such pairs
of forests.
If the exchange moves the node labeled $1$ then it must exchange
it with another part of $X$
since $1$ is the smallest node label in $X$.
Then $\left[X\right]=\left[U\right]$. Otherwise $U=X'\bullet V$
for some forest $V$ such that $V\sim Y$.
This shows that $\left[U\right]$ is a term of \autoref{IotaFactorization}.
\end{proof}

\section{Proof of the quiver}\label{FactorizationSection}
In this section we prove that $Q_n$ is the ordinary quiver of the algebra
$\Sigma\left(\sym{n}\right)$. We begin with a construction that produces
an element $\F\left(X\right)\in L^+$
such that $\E\left(\F\left(X\right)\right)=X$
for all $X\in M^+$.
While this can be done by simply labeling the nodes of
$X$ in any legitimate way,
the labeling provided by $\F$ is convenient
in the proofs of the following results.
If $X$ has length zero, then $X$ is also in $L^+$ and
we take $\F\left(X\right)=X$.
Otherwise suppose that $X=X_1X_2\cdots X_j$
where $X_1,X_2,\ldots,X_j$ are trees.
Let $i$ be minimal such that $\ell\left(X_i\right)>0$
and let $X_{i1},X_{i2}$ be trees such that
$X_i=\vcenter{\begin{xy}<.3cm,0cm>:
(1,2)="1"*+!U{_i};
"1";(0,1)*+!U{X_{i1}}**\dir{-};
"1";(2,1)*+!U{X_{i2}}**\dir{-};
\end{xy}}$. Let
$Y$ be the forest obtained from
$X_1X_2\cdots X_{i-1}X_{i1}X_{i2}X_{i+1}\cdots X_j$
by reducing all the node labels by one
and write $x_1x_2\cdots x_{i-1}x_{i1}x_{i2}x_{i+1}\cdots x_j=\overline{Y}$.
Then defining
\[\F\left(X\right)=
\left(\vcenter{\begin{xy}<.3cm,0cm>:
(0,2)*+!UR{{\color{red}x_1x_2}\cdots{\color{red}x_{i-1}}};
(2,2)="1"*+!U{_1};
"1";(1,1)*+!U{{\color{red}x_{i1}}}**\dir{-};
"1";(3,1)*+!U{{\color{red}x_{i2}}}**\dir{-};
(4,2)*+!UL{{\color{red}x_{i+1}}\cdots{\color{red}x_j}};
\end{xy}}\right)\bullet\F\left(Y\right)\]
we have $\E\left(\F\left(X\right)\right)=X$ by induction.
Note that the nodes of any part of $\F\left(X\right)$ are labeled in prefix
order and are smaller than those in the following part. For example, if
\[X=\vcenter{\begin{xy}<.2cm,0cm>:
(6,4)="1";
"1";(2,3)="2"**\dir{-};
"2";(1,2)*+!U{\color{red}1}**\dir{-};
"2";(4,2)="3"**\dir{-};
"3";(3,1)*+!U{\color{red}1}**\dir{-};
"3";(5,1)*+!U{\color{red}2}**\dir{-};
"1";(8,3)="4"**\dir{-};
"4";(7,2)*+!U{\color{red}1}**\dir{-};
"4";(9,2)*+!U{\color{red}4}**\dir{-};
(13,4)="5";
"5";(11,3)="6"**\dir{-};
"6";(10,2)*+!U{\color{red}1}**\dir{-};
"6";(12,2)*+!U{\color{red}2}**\dir{-};
"5";(17,3)="7"**\dir{-};
"7";(15,2)="8"**\dir{-};
"8";(14,1)*+!U{\color{red}1}**\dir{-};
"8";(16,1)*+!U{\color{red}2}**\dir{-};
"7";(18,2)*+!U{\color{red}4}**\dir{-};
\end{xy}}
\qquad\text{then}\qquad
\F\left(X\right)=\vcenter{\begin{xy}<.2cm,0cm>:
(6,4)="1"*+!U{_1};
"1";(2,3)="2"*+!U{_2}**\dir{-};
"2";(1,2)*+!U{\color{red}1}**\dir{-};
"2";(4,2)="3"*+!U{_3}**\dir{-};
"3";(3,1)*+!U{\color{red}1}**\dir{-};
"3";(5,1)*+!U{\color{red}2}**\dir{-};
"1";(8,3)="4"*+!U{_4}**\dir{-};
"4";(7,2)*+!U{\color{red}1}**\dir{-};
"4";(9,2)*+!U{\color{red}4}**\dir{-};
(13,4)="5"*+!U{_5};
"5";(11,3)="6"*+!U{_6}**\dir{-};
"6";(10,2)*+!U{\color{red}1}**\dir{-};
"6";(12,2)*+!U{\color{red}2}**\dir{-};
"5";(17,3)="7"*+!U{_7}**\dir{-};
"7";(15,2)="8"*+!U{_8}**\dir{-};
"8";(14,1)*+!U{\color{red}1}**\dir{-};
"8";(16,1)*+!U{\color{red}2}**\dir{-};
"7";(18,2)*+!U{\color{red}4}**\dir{-};
\end{xy}}.\]

Next we introduce a total order $<$ on the set of unlabeled
trees. Let $X$ and $Y$ be unlabeled trees.
If $\ell\left(X\right)>0$ then let $X_1,X_2$ be trees such that
$X=\vcenter{\begin{xy}<.25cm,0cm>:
(1,2)="1";
"1";(0,1)*+!U{X_1}**\dir{-};
"1";(2,1)*+!U{X_2}**\dir{-};
\end{xy}}$
and similarly for $Y$.
We write $X<Y$ if one
of the following conditions holds.
\begin{enumerate}
\item $\overline{X}<\overline{Y}$
\item $\overline{X}=\overline{Y}$ and $\ell\left(X\right)>\ell\left(Y\right)$
\item\label{Condition3}
$\overline{X}=\overline{Y}$ and $\ell\left(X\right)=\ell\left(Y\right)$
and $X_1<Y_1$
\item\label{Condition4}
$\overline{X}=\overline{Y}$ and $\ell\left(X\right)=\ell\left(Y\right)$
and $X_1=Y_1$ and $X_2<Y_2$
\end{enumerate}
Note that in situations \autoref{Condition3} and \autoref{Condition4}
the trees $X_1,X_2,Y_1,Y_2$ have length shorter than
$\ell\left(X\right)=\ell\left(Y\right)$ and can therefore be compared
by induction.
For example, the following trees appear in increasing order.
\begin{multline*}
\vcenter{\begin{xy}<.2cm,0cm>:
(2,5)="1";
"1";(1,4)*+!U{\color{red}1}**\dir{-};
"1";(4,4)="2"**\dir{-};
"2";(3,3)*+!U{\color{red}2}**\dir{-};
"2";(6,3)="3"**\dir{-};
"3";(5,2)*+!U{\color{red}3}**\dir{-};
"3";(8,2)="4"**\dir{-};
"4";(7,1)*+!U{\color{red}1}**\dir{-};
"4";(9,1)*+!U{\color{red}3}**\dir{-};
\end{xy}}
<
\vcenter{\begin{xy}<.2cm,0cm>:
(4,4)="1";
"1";(2,3)="2"**\dir{-};
"2";(1,2)*+!U{\color{red}1}**\dir{-};
"2";(3,2)*+!U{\color{red}2}**\dir{-};
"1";(6,3)="3"**\dir{-};
"3";(5,2)*+!U{\color{red}3}**\dir{-};
"3";(8,2)="4"**\dir{-};
"4";(7,1)*+!U{\color{red}1}**\dir{-};
"4";(9,1)*+!U{\color{red}3}**\dir{-};
\end{xy}}
<
\vcenter{\begin{xy}<.2cm,0cm>:
(2,4)="1";
"1";(1,3)*+!U{\color{red}3}**\dir{-};
"1";(6,3)="2"**\dir{-};
"2";(4,2)="3"**\dir{-};
"3";(3,1)*+!U{\color{red}1}**\dir{-};
"3";(5,1)*+!U{\color{red}2}**\dir{-};
"2";(8,2)="4"**\dir{-};
"4";(7,1)*+!U{\color{red}1}**\dir{-};
"4";(9,1)*+!U{\color{red}3}**\dir{-};
\end{xy}}
<
\vcenter{\begin{xy}<.2cm,0cm>:
(2,5)="1";
"1";(1,4)*+!U{\color{red}3}**\dir{-};
"1";(4,4)="2"**\dir{-};
"2";(3,3)*+!U{\color{red}3}**\dir{-};
"2";(6,3)="3"**\dir{-};
"3";(5,2)*+!U{\color{red}1}**\dir{-};
"3";(8,2)="4"**\dir{-};
"4";(7,1)*+!U{\color{red}1}**\dir{-};
"4";(9,1)*+!U{\color{red}2}**\dir{-};
\end{xy}}
<
\vcenter{\begin{xy}<.2cm,0cm>:
(4,4)="1";
"1";(2,3)="2"**\dir{-};
"2";(1,2)*+!U{\color{red}1}**\dir{-};
"2";(3,2)*+!U{\color{red}3}**\dir{-};
"1";(6,3)="3"**\dir{-};
"3";(5,2)*+!U{\color{red}1}**\dir{-};
"3";(8,2)="4"**\dir{-};
"4";(7,1)*+!U{\color{red}2}**\dir{-};
"4";(9,1)*+!U{\color{red}3}**\dir{-};
\end{xy}}
\end{multline*}

The relation $<$ induces the lexicographic order on unlabeled forests,
which we also denote by $<$.
This allows us to introduce the notion of a
{\em nondecreasing representative} $X\in M$
of its class $\left[X\right]$,
namely the element whose parts appear in nondecreasing order
from left to right.
The most important property of the nondecreasing representative
is given in the following lemma.

\begin{lemma}\label{FMinimum} If $X\in M^+$ is nondecreasing
and $Z\in L^+$ is such that
$Z\sim\F\left(X\right)$ but
$\left[Z\right]\ne\left[\F\left(X\right)\right]$
then $\E\left(Z\right)<X$.
\end{lemma}

\begin{proof}
Let $p.\lb{a_1}{b_1}\cdots\rb{a_l}{b_l}$
be the path $\p\left(\F\left(X\right)\right)$ in $Q$ with destination
$p=p_1p_2\cdots p_j$.
Then any $Z\in L^+$ such that
$Z\sim\F\left(X\right)$ can be assembled from the set
\begin{equation}\label{Steckdosen}
{\color{red}p_1},{\color{red}p_2},\ldots,{\color{red}p_j},
\vcenter{\begin{xy}<.3cm,0cm>:
(1,2)="1"*+!U_{_1};
"1";(0,1)*+!U{\color{red}a_1}**\dir{-};
"1";(2,1)*+!U{\color{red}b_1}**\dir{-};
\end{xy}},\ldots,
\vcenter{\begin{xy}<.3cm,0cm>:
(1,2)="1"*+!U_{_l};
"1";(0,1)*+!U{\color{red}a_l}**\dir{-};
"1";(2,1)*+!U{\color{red}b_l}**\dir{-};
\end{xy}}
\end{equation}
by identifying the tree
$\vcenter{\begin{xy}<.25cm,0cm>:
(1,2)="1"*+!U_{_i};
"1";(0,1)*+!U{\color{red}a_i}**\dir{-};
"1";(2,1)*+!U{\color{red}b_i}**\dir{-};
\end{xy}}$
in \autoref{Steckdosen} with
one of the leaves of value $a_i+b_i$ in \autoref{Steckdosen}
for all $1\le i\le l$.
This sequence of identifications can in turn be interpreted
as an injective function $\left\{1,2,\ldots,l\right\}\to
\left\{1,2,\ldots,j+2l\right\}$.
Viewing $\F\left(X\right)$ and $Z$ as injective functions,
the sequence of exchanges of subtrees of equal
squash transforming $\F\left(X\right)$ into $Z$ is equivalent to
a permutation of $\left\{1,2,\ldots,j+2l\right\}$.
We can express this permutation as a product
of disjoint cycles. In terms of forests, each of these cycles permutes
a set of subtrees of $\F\left(X\right)$ of equal squash.
Note that the set of trees permuted by such a cycle
contains at most one leaf, since we regard leaves 
of the same value as indistinguishable
when decomposing a permutation into disjoint cycles.

Since these cycles act on disjoint sets of subtrees,
we can assume that the sequence of subtree exchanges 
transforming $\F\left(X\right)$ into $Z$ is a single cycle permuting
subtrees of the same squash, at most one of which being a leaf.
Suppose that the cycle moves the subtree $U$ of positive length
to the position of the subtree $V$.
If $V$ has no parent, then it lies to the left of $U$ since the parts
of $X$ appear in nondecreasing order.
If $V$ has a parent, then again $V$ lies to the left
of $U$ since otherwise the parent of $V$
would have a larger node label than $U$.
We conclude that the leftmost subtree permuted by the cycle is
a leaf and that the subtrees of positive length all move to the left,
resulting in a forest which under $\E$ is
lexicographically smaller than $X$.
\end{proof}

Assembling the results above yields the following main results of this section.

\begin{proposition}\label{MainFactorization}
$\E\circ\iota:kQ\to k\mathcal{M}^+$ is surjective.
\end{proposition}

\begin{proof}
Let $X$ be a nondecreasing element of $M^+$
and put $P=\p\left[\F\left(X\right)\right]$.
Then $\iota\left(P\right)
=\sum_{\left[U\right]\sim\left[\F\left(X\right)\right]}\left[U\right]$
by \autoref{IotaWiggle} 
so that taking $\mathcal{Y}
=\E\left(\iota\left(P\right)-\left[\F\left(X\right)\right]\right)$
we have $\left[Y\right]<\left[X\right]$
for each term $\left[Y\right]$ of $\mathcal{Y}$
by \autoref{FMinimum}.
Repeating the argument for all the terms of $\mathcal{Y}$
and subtracting the result from $P$ results in an element of $kQ$
mapping to $\left[X\right]$ under $\E\circ\iota$ by induction.
\end{proof}

For example,
applying \autoref{MainFactorization} to
$\left[\vcenter{\begin{xy}<.2cm,0cm>:
(2,3)="1"*+!U{_1};
"1";(1,2)*+!U{\color{red}1}**\dir{-};
"1";(4,2)="2"*+!U{_2}**\dir{-};
"2";(3,1)*+!U{\color{red}1}**\dir{-};
"2";(5,1)*+!U{\color{red}2}**\dir{-};
(6,3)*+!U{\color{red}3};
\end{xy}}\right]$
produces the path $p.\lb{1}{3}\rb{1}{2}-p.\lb{1}{2}\rb{1}{3}$
where $p$ is the partition $34$.

\begin{corollary}\label{IotaSurjectiveModKernel}
$\iota$ is surjective modulo $\ker\Delta$.
\end{corollary}

\begin{proof}
Let $X\in L$. We will show that some element of $kQ$
maps under $\iota$ to an element of $k\mathcal{L}^+$
congruent to $\left[X\right]$ modulo $\ker\Delta$.
If $X$ has a node
$\vcenter{\begin{xy}<.25cm,0cm>:
(1,2)="1"*+!U_{_i};
"1";(0,1)*+!U{Z_1}**\dir{-};
"1";(2,1)*+!U{Z_2}**\dir{-};
\end{xy}}$ for which
$Z_1=Z_2\in\mathbb{N}$ then $\left[X\right]\in\ker\Delta$
and we can take $P=0$.
Otherwise by \autoref{AlignedRendering} applied to all the terms
of $\E\left[X\right]$ there exist
$A\in kM^+$ and $\mathcal{Y}\in\mathcal{N}+\mathcal{J}$ such that
$\E\left[X\right]=A+\mathcal{Y}$. 
Applying $\F$ we have
$\left[X\right]\equiv\F\left(A\right)\pmod{\ker\Delta}$.
By \autoref{MainFactorization}
we have $P\in kQ$ such that $\E\left(\iota\left(P\right)\right)=A$
so that $\iota\left(P\right)-\F\left(A\right)
\in\ker\E\subseteq\ker\Delta$.
It follows that $\iota\left(P\right)\equiv
\F\left(A\right)\equiv\left[X\right]\pmod{\ker\Delta}$.
\end{proof}

\begin{proposition}\label{ExtQuiver}
$Q_n$ is the ordinary quiver of $\Sigma\left(\sym{n}\right)$.
\end{proposition}

\begin{proof} Let $I=\iota^{-1}\left(\ker\Delta\right)$
so that $kQ_n/I\cong\iota\left(kQ_n\right)/\ker\Delta$ since $\iota$
is injective by \autoref{IotaInjective}. But
$\iota\left(kQ_n\right)/\ker\Delta
\cong k\mathcal{L}_n/\ker\Delta$
by \autoref{IotaSurjectiveModKernel} and
$k\mathcal{L}_n/\ker\Delta\cong
\Sigma\left(\sym{n}\right)^\mathsf{op}$ by \autoref{LIsomorphism}.
Let $R$ be the Jacobson radical of $kQ_n$. Then $R$ is generated
by all paths of $Q_n$ of positive length.
Since $Q_n$ is the ordinary quiver of any quotient of $kQ_n$ by
an ideal contained in $R^2$ by \cite[Lemma~3.6]{blue}
it suffices to show that $I\subseteq R^2$.

Let $P$ be any element of $I$.
By multiplying $P$ on the left and on the right
by various vertices of $Q_n$ we can split $P$ into a sum
of elements of $I$ all of whose terms have the same source
and destination. We therefore assume that
all the terms of $P$ have the same source and destination and hence the
same length. If this length were zero or one, then 
$P$ would be a multiple of a vertex or an edge.
But $\Delta\left(\iota\left(p\right)\right)=p\ne 0$
for all vertices $p$, while
$\Delta\left(\iota\left(e\right)\right)\ne 0$
for all edges $e$ by \autoref{NodesInImage}. Therefore
$P\in R^2$.
\end{proof}

\section{The relations}\label{RelationsSection}
In this section we state our conjecture
on the relations for the quiver presentation of
$\Sigma\left(\sym{n}\right)$.
Let $\mathcal{R}\subseteq k\mathcal{B}^\ast$ be the set of elements
\begin{equation}\label{2WayNotPrec}
\lb{a}{b}\rb{c}{d}-\lb{c}{d}\rb{a}{b}\qquad\text{where}\qquad
a+b\not\in\left\{c,d\right\}
\qquad\text{and}\qquad
c+d\not\in\left\{a,b\right\}
\end{equation}
and the elements
\begin{equation}
\label{BranchRelation}
\lb{a}{b}\mb{c}{d}\rb{x}{y}
+\lb{x}{y}\mb{a}{b}\rb{c}{d}
-\lb{a}{b}\mb{x}{y}\rb{c}{d}
-\lb{c}{d}\mb{x}{y}\rb{a}{b}
\end{equation}
where $a,b,c,d$ satisfy the condition in \autoref{2WayNotPrec} and either
\begin{enumerate}
\item\label{RightSpaceship} $a+b=c+d\in\left\{x,y\right\}$ or
\item\label{LeftSpaceship} $x+y\in\left\{a,b\right\}\cap\left\{c,d\right\}$.
\end{enumerate}
The elements of $\mathcal{R}$ are called {\em branch relations}.
The following proposition shows that the branch relations
produce relations when applied to vertices of $Q$.

\begin{proposition}\label{BranchProposition}
If $R\in\mathcal{R}$ then $p.R\in\ker\left(\E\circ\iota\right)$
for all vertices $p$ of $Q$.
\end{proposition}

\begin{proof}
Suppose $R=\lb{a}{b}\rb{c}{d}-\lb{c}{d}\rb{a}{b}$
where $a,b,c,d\in\mathbb{N}$ satisfy the condition in
\autoref{2WayNotPrec}. Then
for any partition $p$ we have
$\iota\left(p.\lb{a}{b}\rb{c}{d}\right)
=\left[\vcenter{\begin{xy}<.2cm,0cm>:
(2,2)="1"*+!U{_1};
"1";(1,1)*+!U{\color{red}a}**\dir{-};
"1";(3,1)*+!U{\color{red}b}**\dir{-};
(5,2)="1"*+!U{_2};
"1";(4,1)*+!U{\color{red}c}**\dir{-};
"1";(6,1)*+!U{\color{red}d}**\dir{-};
(7,2)*+!UL{{\color{red}q_1}\cdots{\color{red}q_j}};
\end{xy}}\right]$ and
$\iota\left(p.\lb{c}{d}\rb{a}{b}\right)
=\left[\vcenter{\begin{xy}<.2cm,0cm>:
(2,2)="1"*+!U{_2};
"1";(1,1)*+!U{\color{red}a}**\dir{-};
"1";(3,1)*+!U{\color{red}b}**\dir{-};
(5,2)="1"*+!U{_1};
"1";(4,1)*+!U{\color{red}c}**\dir{-};
"1";(6,1)*+!U{\color{red}d}**\dir{-};
(7,2)*+!UL{{\color{red}q_1}\cdots{\color{red}q_j}};
\end{xy}}\right]$
if $p$ has parts $a+b$ and $c+d$ where $q_1,\ldots,q_j$
are the remaining parts of $p$,
while both expressions are zero if $p$ has no
part $a+b$ or no part $c+d$.
This shows that $\E\left(\iota\left(p.R\right)\right)=0$
for all vertices $p$.

Now let $R$ be the element in \autoref{BranchRelation}
and suppose $a,b,c,d,x,y\in\mathbb{N}$ satisfy
condition~\autoref{RightSpaceship} of
the definition of $\mathcal{R}$.
Specifically, we assume that $a+b=c+d=x$ although
the argument can be easily modified if $a+b=c+d=y$.
In each of the cases that
\begin{enumerate}
\item $p$ has at least one part $x+y$ and exactly one part $x$
\item $p$ has at least one part $x+y$ and two or more parts $x$
\item $p$ has no part $x+y$ or no part $x$
\end{enumerate}
the image of $p.R$ can be calculated explicitly.
In the third case all four terms
of $p.R$ are zero.
In the second case we take $p$ to be the
partition with parts
$x+y,x,x,q_1,\ldots,q_j$ for any $q_1,\ldots,q_j\in\mathbb{N}$.
We calculate
\begin{align*}\textstyle
p.\lb{x}{y}\mb{a}{b}\rb{c}{d}
&\stackrel{\E\circ\iota}{\longrightarrow}
\left[\vcenter{\begin{xy}<.2cm,0cm>:
(4,3)="1";
"1";(2,2)="2"**\dir{-};
"2";(1,1)*+!U{\color{red}a}**\dir{-};
"2";(3,1)*+!U{\color{red}b}**\dir{-};
"1";(5,2)*+!U{\color{red}y}**\dir{-};
(7,3)="1";
"1";(6,2)*+!U{\color{red}c}**\dir{-};
"1";(8,2)*+!U{\color{red}d}**\dir{-};
(9,3)*+!UL{{\color{red}xq_1}\cdots{\color{red}q_j}};
\end{xy}}\right]
+\left[\vcenter{\begin{xy}<.2cm,0cm>:
(4,3)="1";
"1";(2,2)="2"**\dir{-};
"2";(1,1)*+!U{\color{red}c}**\dir{-};
"2";(3,1)*+!U{\color{red}d}**\dir{-};
"1";(5,2)*+!U{\color{red}y}**\dir{-};
(7,3)="1";
"1";(6,2)*+!U{\color{red}a}**\dir{-};
"1";(8,2)*+!U{\color{red}b}**\dir{-};
(9,3)*+!UL{{\color{red}xq_1}\cdots{\color{red}q_j}};
\end{xy}}\right]
+\left[\vcenter{\begin{xy}<.2cm,0cm>:
(2,2)="1";
"1";(1,1)*+!U{\color{red}x}**\dir{-};
"1";(3,1)*+!U{\color{red}y}**\dir{-};
(5,2)="1";
"1";(4,1)*+!U{\color{red}a}**\dir{-};
"1";(6,1)*+!U{\color{red}b}**\dir{-};
(8,2)="1";
"1";(7,1)*+!U{\color{red}c}**\dir{-};
"1";(9,1)*+!U{\color{red}d}**\dir{-};
(9,2)*+!UL{{\color{red}q_1}\cdots{\color{red}q_j}};
\end{xy}}\right]\\
\textstyle p.\lb{a}{b}\mb{x}{y}\rb{c}{d}
&\stackrel{\E\circ\iota}{\longrightarrow}
\left[\vcenter{\begin{xy}<.2cm,0cm>:
(4,3)="1";
"1";(2,2)="2"**\dir{-};
"2";(1,1)*+!U{\color{red}c}**\dir{-};
"2";(3,1)*+!U{\color{red}d}**\dir{-};
"1";(5,2)*+!U{\color{red}y}**\dir{-};
(7,3)="1";
"1";(6,2)*+!U{\color{red}a}**\dir{-};
"1";(8,2)*+!U{\color{red}b}**\dir{-};
(9,3)*+!UL{{\color{red}xq_1}\cdots{\color{red}q_j}};
\end{xy}}\right]
+\left[\vcenter{\begin{xy}<.2cm,0cm>:
(2,2)="1";
"1";(1,1)*+!U{\color{red}x}**\dir{-};
"1";(3,1)*+!U{\color{red}y}**\dir{-};
(5,2)="1";
"1";(4,1)*+!U{\color{red}a}**\dir{-};
"1";(6,1)*+!U{\color{red}b}**\dir{-};
(8,2)="1";
"1";(7,1)*+!U{\color{red}c}**\dir{-};
"1";(9,1)*+!U{\color{red}d}**\dir{-};
(9,2)*+!UL{{\color{red}q_1}\cdots{\color{red}q_j}};
\end{xy}}\right]\\
\textstyle p.\lb{c}{d}\mb{x}{y}\rb{a}{b}
&\stackrel{\E\circ\iota}{\longrightarrow}
\left[\vcenter{\begin{xy}<.2cm,0cm>:
(4,3)="1";
"1";(2,2)="2"**\dir{-};
"2";(1,1)*+!U{\color{red}a}**\dir{-};
"2";(3,1)*+!U{\color{red}b}**\dir{-};
"1";(5,2)*+!U{\color{red}y}**\dir{-};
(7,3)="1";
"1";(6,2)*+!U{\color{red}c}**\dir{-};
"1";(8,2)*+!U{\color{red}d}**\dir{-};
(9,3)*+!UL{{\color{red}xq_1}\cdots{\color{red}q_j}};
\end{xy}}\right]
+\left[\vcenter{\begin{xy}<.2cm,0cm>:
(2,2)="1";
"1";(1,1)*+!U{\color{red}x}**\dir{-};
"1";(3,1)*+!U{\color{red}y}**\dir{-};
(5,2)="1";
"1";(4,1)*+!U{\color{red}a}**\dir{-};
"1";(6,1)*+!U{\color{red}b}**\dir{-};
(8,2)="1";
"1";(7,1)*+!U{\color{red}c}**\dir{-};
"1";(9,1)*+!U{\color{red}d}**\dir{-};
(9,2)*+!UL{{\color{red}q_1}\cdots{\color{red}q_j}};
\end{xy}}\right]\\
\textstyle p.\lb{a}{b}\mb{c}{d}\rb{x}{y}
&\stackrel{\E\circ\iota}{\longrightarrow}
\left[\vcenter{\begin{xy}<.2cm,0cm>:
(2,2)="1";
"1";(1,1)*+!U{\color{red}x}**\dir{-};
"1";(3,1)*+!U{\color{red}y}**\dir{-};
(5,2)="1";
"1";(4,1)*+!U{\color{red}a}**\dir{-};
"1";(6,1)*+!U{\color{red}b}**\dir{-};
(8,2)="1";
"1";(7,1)*+!U{\color{red}c}**\dir{-};
"1";(9,1)*+!U{\color{red}d}**\dir{-};
(9,2)*+!UL{{\color{red}q_1}\cdots{\color{red}q_j}};
\end{xy}}\right]
\end{align*}
 
so that $\E\left(\iota\left(p.R\right)\right)=0$.
The first case is similar to the second and the calculation in
the case that $a,b,c,d,x,y$ satisfy
condition~\autoref{LeftSpaceship} of the definition of $\mathcal{R}$
is similar to the calculation above.
\end{proof}

For unlabeled trees $X,Y,Z$ we denote
$\vcenter{\begin{xy}<.2cm,0cm>:
(2,3)="1";
"1";(1,2)*+!U{X}**\dir{-};
"1";(4,2)="2"**\dir{-};
"2";(3,1)*+!U{Y}**\dir{-};
"2";(5,1)*+!U{Z}**\dir{-};
\end{xy}}
+\vcenter{\begin{xy}<.2cm,0cm>:
(2,3)="1";
"1";(1,2)*+!U{Z}**\dir{-};
"1";(4,2)="2"**\dir{-};
"2";(3,1)*+!U{X}**\dir{-};
"2";(5,1)*+!U{Y}**\dir{-};
\end{xy}}
+\vcenter{\begin{xy}<.2cm,0cm>:
(2,3)="1";
"1";(1,2)*+!U{Y}**\dir{-};
"1";(4,2)="2"**\dir{-};
"2";(3,1)*+!U{Z}**\dir{-};
"2";(5,1)*+!U{X}**\dir{-};
\end{xy}}$
by $\J\left(X,Y,Z\right)$.
Suppose that $A=\sum_{i=1}^mA_i$ is an aligned
rendering of $\J\left(X,Y,Z\right)$ where
$A_1,A_2,\ldots,A_m$ are aligned unlabeled trees.
If it exists, $A$ may have
more or fewer than three terms and satisfies
$A-\J\left(X,Y,Z\right)\in\ker\pi$
by \autoref{AlignedRendering}.
But since $\J\left(X,Y,Z\right)\in\ker\pi$ we have $A\in\ker\pi$.
For any unlabeled forests $U$ and $V$
we denote by $UV$ the forest whose parts are the parts of $U$
followed by the parts of $V$.
This applies in particular when $V$ is a composition,
which we regard as a forest of length zero.
For any composition $q$ we
apply \autoref{MainFactorization}
to each class $\left[A_iq\right]$.
This produces an element $P\in kQ$ such that
$\E\left(\iota\left(P\right)\right)=\sum_{i=1}^m\left[A_iq\right]$.
Then $P\in\ker\left(\Delta\circ\iota\right)$.

Let $\mathcal{S}$ be the set of elements $P\in kQ$
for which $\E\left(\iota\left(P\right)\right)
=\sum_{i=1}^m\left[A_iq\right]$ for some
composition $q$ and some aligned rendering
$\sum_{i=1}^mA_i$ of an element of the form
\begin{enumerate}
\item\label{SmallJacobi}
$\J\left({\color{red}x},{\color{red}y},{\color{red}z}\right)$
where $x<y<z$ are natural numbers such that $x+y\ne z$, or
\item\label{BigJacobi}
$\J\left(\vcenter{\begin{xy}<.2cm,0cm>:
(2,2)="1";
"1";(1,1)*+!U{\color{red}x_1}**\dir{-};
"1";(3,1)*+!U{\color{red}x_2}**\dir{-};
\end{xy}},
{\color{red}y},{\color{red}z}\right)$
where $x_1<x_2$ and $y<z$ are natural numbers such that
$x_1+x_2\in\left\{y,z,y+z\right\}$.
\end{enumerate}
Then the elements of $\mathcal{S}$ are relations for the
quiver presentation of $\Sigma\left(\sym{n}\right)$.
Observe that elements of the form \autoref{SmallJacobi}
have only one possible aligned
rendering, while elements of the form \autoref{BigJacobi}
have only one ``useful'' aligned rendering.
For example, the term
$\vcenter{\begin{xy}<.2cm,0cm>:
(2,4)="1";
"1";(1,3)*+!U{\color{red}6}**\dir{-};
"1";(6,3)="2"**\dir{-};
"2";(4,2)="3"**\dir{-};
"3";(3,1)*+!U{\color{red}1}**\dir{-};
"3";(5,1)*+!U{\color{red}2}**\dir{-};
"2";(7,2)*+!U{\color{red}3}**\dir{-};
\end{xy}}$
of $\J\left(
\vcenter{\begin{xy}<.2cm,0cm>:
(2,2)="1";
"1";(1,1)*+!U{\color{red}1}**\dir{-};
"1";(3,1)*+!U{\color{red}2}**\dir{-};
\end{xy}},
{\color{red}3},{\color{red}6}\right)$ has
the two aligned renderings shown in \autoref{TwoAlignedRenderings}
but only the second can be used to construct a relation,
since the terms of the first aligned rendering cancel
the other terms of $\J\left(
\vcenter{\begin{xy}<.2cm,0cm>:
(2,2)="1";
"1";(1,1)*+!U{\color{red}1}**\dir{-};
"1";(3,1)*+!U{\color{red}2}**\dir{-};
\end{xy}},
{\color{red}3},{\color{red}6}\right)$.

We conjecture that the relations above generate the ideal
of relations for the quiver presentation of $\Sigma\left(\sym{n}\right)$
in the following way.

\begin{conjecture}\label{MainConjecture}
The descent algebra $\Sigma\left(\sym{n}\right)$ has a presentation
as the path algebra $kQ_n$ subject to the relations
$\mathcal{S}\cap kQ_n$ and $p.R$ 
for all partitions $p$ of $n$ and
all $R\in\mathcal{R}$.
In particular, the relations all have length two or three.
\end{conjecture}

We have verified \autoref{MainConjecture} through a computer
calculation for $n\le 15$.
In fact, we have implemented a procedure in \textsf{GAP} \cite{GAP4}
which calculates minimal projective resolutions over
the algebra $A=kQ_n/\ker\left(\Delta\circ\iota\right)$ of the simple module
$\left(A/\rad{A}\right)p$ for all partitions $p$ of $n$.
One result of the calculation is a minimal generating set
of $\ker\left(\Delta\circ\iota\right)$ which can be used to confirm
that the presentation in \autoref{MainConjecture} is correct
for small~$n$.

\autoref{Statistics} shows the minimal number of relations
for the presentation of $\Sigma\left(\sym{n}\right)$ for $n\le 15$.
The table also shows the numbers of branch and Jacobi relations.
Note that when $n\ge 10$ the total number of
branch and Jacobi relations exceeds
the size of the minimal generating set.
Nonetheless, the ideal generated by
the branch and Jacobi relations is exactly
$\ker\left(\Delta\circ\iota\right)$ in every case shown in
\autoref{Statistics}.

\begin{table}
\caption{Numbers of Relations}\label{Statistics}
\[\begin{array}{cccc}
n&\text{Branch}&\text{Jacobi}&\text{Minimal}\\\hline
6&0&1&1\\
7&1&3&4\\
8&4&7&11\\
9&10&14&24\\
10&22&29&48\\
11&44&51&90\\
12&86&89&160\\
13&152&146&270\\
14&265&240&444\\
15&441&369&705
\end{array}\]
\end{table}

\section{Example}\label{ExampleSection}
As an example of \autoref{MainConjecture}
we calculate the presentation for $\Sigma\left(\sym{8}\right)$.
The quiver for this presentation is shown in \autoref{Q8}.
To calculate the branch relations, we list
all $R\in\mathcal{R}$
and apply them to all vertices $p$ of $Q_8$.
Those resulting in nonzero relations are shown in the column labeled $P$
of \autoref{Q8Table}.
To calculate the Jacobi relations, we list all tuples
$x,y,z$ and $x_1,x_2,y,z$ satisfying the conditions in
the definition of $\mathcal{S}$. For each partition $q=q_1q_2\cdots q_j$
completing $x+y+z$ or $x_1+x_2+y+z$ to
a composition $p$ of $n$ we find an element $P\in kQ$
for which $\E\left(\iota\left(P\right)\right)
=\sum_{i=1}^m\left[A_iq\right]$ where
$\sum_{i=1}^mA_i$ is an aligned rendering of 
$\J\left({\color{red}x},{\color{red}y},{\color{red}z}\right)$
or $\J\left(\vcenter{\begin{xy}<.2cm,0cm>:
(2,2)="1";
"1";(1,1)*+!U{\color{red}x_1}**\dir{-};
"1";(3,1)*+!U{\color{red}x_2}**\dir{-};
\end{xy}},
{\color{red}y},{\color{red}z}\right)$.
These relations are also shown in \autoref{Q8Table}.

\begin{table}
\caption{Relations for $\Sigma\left(\sym{8}\right)$}\label{Q8Table}
\renewcommand{\arraystretch}{1.5}
\[\begin{array}{rcl}
p&R&P\\\hline
35&\lb{1}{2}\rb{1}{4}-\lb{1}{4}\rb{1}{2}&jq-kt\\
134&\lb{1}{2}\mb{1}{2}\rb{1}{3}
+\lb{1}{3}\mb{1}{2}\rb{1}{2}
-2\lb{1}{2}\mb{1}{3}\rb{1}{2}
&adm-2ack\\
35&\lb{1}{2}\mb{1}{2}\rb{2}{3}
+\lb{2}{3}\mb{1}{2}\rb{1}{2}
-2\lb{1}{2}\mb{2}{3}\rb{1}{2}
&bis-2bhq\\
44&\lb{1}{3}\mb{1}{3}\rb{1}{2}
+\lb{1}{2}\mb{1}{3}\rb{1}{3}
-2\lb{1}{3}\mb{1}{2}\rb{1}{3}
&dmu-2cku\\
116&\J\left(\vcenter{\begin{xy}<.2cm,0cm>:
(2,2)="1";
"1";(1,1)*+!U{\color{red}1}**\dir{-};
"1";(3,1)*+!U{\color{red}2}**\dir{-};
\end{xy}},
{\color{red}1},{\color{red}2}\right)
&acl-aen\\
17&\J\left({\color{red}1},{\color{red}2},{\color{red}4}\right)
&jr+kv-lw\\
17&\J\left(\vcenter{\begin{xy}<.2cm,0cm>:
(2,2)="1";
"1";(1,1)*+!U{\color{red}1}**\dir{-};
"1";(3,1)*+!U{\color{red}2}**\dir{-};
\end{xy}},
{\color{red}1},{\color{red}3}\right)
&2ckv+clw-dmv-enw\\
26&\J\left(\vcenter{\begin{xy}<.2cm,0cm>:
(2,2)="1";
"1";(1,1)*+!U{\color{red}1}**\dir{-};
"1";(3,1)*+!U{\color{red}2}**\dir{-};
\end{xy}},
{\color{red}1},{\color{red}2}\right)
&bgo-bhp\\
8&\J\left({\color{red}1},{\color{red}2},{\color{red}5}\right)
&px+qy-rz\\
8&\J\left(\vcenter{\begin{xy}<.2cm,0cm>:
(2,2)="1";
"1";(1,1)*+!U{\color{red}1}**\dir{-};
"1";(3,1)*+!U{\color{red}3}**\dir{-};
\end{xy}},
{\color{red}1},{\color{red}3}\right)
&mty-mvz\\
8&\J\left(\vcenter{\begin{xy}<.2cm,0cm>:
(2,2)="1";
"1";(1,1)*+!U{\color{red}1}**\dir{-};
"1";(3,1)*+!U{\color{red}2}**\dir{-};
\end{xy}},
{\color{red}2},{\color{red}3}\right)
&gox-hpx-2hqy-isy
\end{array}\]
\end{table}

As mentioned in \autoref{RelationsSection}
we have verified through a computer calculation that the quotient of $kQ_8$
by the ideal generated by the elements in \autoref{Q8Table}
is isomorphic to $\Sigma\left(\sym{8}\right)$.

\bibliographystyle{plain}
\bibliography{artikel}
\end{document}